\documentclass[11pt]{amsart}
\usepackage{verbatim}
\usepackage{tikz-cd}
\usepackage{amsmath}
\usepackage{amssymb}
\usepackage{textcomp}
\usepackage{amsthm}
\usepackage{amsfonts}
\usepackage[all]{xy}
\usepackage{mathrsfs}
 \usepackage{graphicx}
 \usepackage{epstopdf}
 \usepackage{float}
 \usepackage{graphicx}
 \usepackage[T1]{fontenc}
  \usepackage{hyperref}
  \usepackage{bbm}
  \usepackage[top= 2.25cm, bottom=2cm, left = 1.5 cm, right= 1.5 cm]{geometry}

\newtheorem{theo}{Th\'eor\'em}[section]

\def\C{\mathbb{C}}
\def\mC{\mathbb{C}^{\times}}

\def\G{\mathcal{G}}
\def\F{\mathcal{F}}
\def\M{\mathcal{M}}
\def\N{\mathbb{N}}

\def\P{\mathcal{P}}
\def\Q{\mathbb{Q}}
\def\K{\mathcal{K}}
\def\R{\mathbb{R}}

\def\Z{\mathbb{Z}}

\def\n{\mathfrak{n}}

\def\line{\overline}

\def\Hom{\mathop{\mathrm{Hom}}\nolimits}
\def\IC{\mathop{\mathrm{IC}}\nolimits}

\def\Irr{\mathop{\mathrm{Irr}}\nolimits}

\def\dim{\mathop{\mathrm{dim}}\nolimits}
\def\rank{\mathop{\mathrm{rank}}\nolimits}

\def\remk{\noindent\textit{Remark:~}}

\newtheorem{prop}[theo]{Proposition}
\newtheorem{prop-def}[theo]{Proposition-Definition}
\newtheorem{def-prop}[theo]{Definition-Proposition}
\newtheorem{cor}[theo]{Corollary}

\newtheorem{lemma}[theo]{Lemma}
\newtheorem{teo}[theo]{Theorem}
\newtheorem{definition}[theo]{Definition}

\newtheorem{conj}[theo]{Conjecture}
\newtheorem{rem}[theo]{Remark}
\newtheorem{notation}[theo]{Notation}
\newtheorem{example}[theo]{Example}

\everymath{\displaystyle} 
\input xy
\xyoption{all}

\title{The characteristic cycles and semi-canonical bases on type $A$ quiver variety}

\author{Taiwang DENG}
\address{ 
Yau Mathematical Sciences Center, Tsinghua University, Haidian District, Beijing, 100084, China.}
\email{dengtaiw@tsinghua.edu.cn}

\author{Bin Xu}
\address{Yau Mathematical Sciences Center and Department of Mathematics \\  Tsinghua University, Beijing, China}
\email{binxu@tsinghua.edu.cn}

\date{}


\begin{document}

\maketitle

\begin{abstract}
In this article we study a conjecture of Geiss-Leclerc-Schr{\"o}er, which is an analogue of a classical conjecture of Lusztig in the Weyl group case. 
It concerns the relation between canonical basis and semi-canonical basis through the characteristic cycles. We formulate an approach to this conjecture and prove it for type $A_2$ quiver. In the general type A case, we reduce the conjecture to show that certain nearby cycles have vanishing Euler characteristic.
\end{abstract}

\tableofcontents

\section{Introduction}

In \cite{GLS05} Geiss-Leclerc-Schr{\"o}er studied Lusztig's semi-canonical basis \cite{lu00} for 
the enveloping algebra $U(\n)$. Here $\n$ is the maximal nilpotent 
subalgebra of some symmetric Kac-Moody Lie algebra over $\C$. They raised the question of the 
relation between the semi-canonical basis, the canonical basis and the singular support (cf. \cite{GLS05}, 1.5), 
referring to a conjecture made by Lusztig for the Weyl group algebra (cf. \cite{lu97}, 4.17).

In this paper, we consider the conjecture of Geiss-Leclerc-Schr{\"o}er mentioned above for the quiver $(I, Q)$ of type $A$ with orientation $\Omega: i \rightarrow i+1$. The variety $E_{V, \Omega}$ of quiver representations in an $I$-graded vector space $V$ admits a stratification by an action of a reductive group $G_V$. For each orbit $S$, we can associate a perverse sheaf $\IC(\line{S}, \C)$. They give rise to a basis $\{g_{S}\}$ for $U(\n)$, called the canonical basis. We By considering the union of the conormal bundles over the orbits on $E_{V, \Omega}$
\[
\Lambda_{V} := \bigcup_{S} T^*_{S}E_{V, \Omega}
\]
Lusztig constructed the semi-canonical basis for $U(\n)$, denoted by $\phi_{S}$, with respect to the irreducible component $\overline{T^*_{S}E_{V, \Omega}}$. Let $m_{S', S} \in \mathbb{C}$ be the coefficients of the expansion of $g_{S}$ with respect to the basis $\{\phi_{S'}\}$, i.e.,
\[
g_{S} = \sum_{S'} m_{S', S} \phi_{S'}.
\]
On the other hand, Kashiwara and Shapira constructed a 
characteristic cycle $CC(\F)$ for a constructible sheaf $\F$ on a manifold (cf. \cite{ka13}), which can be written as
\[
CC(\IC(\line{S}, \C)) = [T^{*}_{S}E_{V, \Omega}]+\sum_{S' \subseteq \line{S}} n_{S', S}[T^{*}_{S'}E_{V, \Omega}], \quad n_{S',S}  \in \mathbb{Z}_{\geq 0}.
\]
Furthermore, they constructed a morphism 
\[
E_u: L(T^*E_{V, \Omega})\rightarrow M(E_{V, \Omega})
\]
where $L(T^*E_{V, \Omega})$ denotes the group of Lagrangian cycles and $M(E_{V, \Omega})$ the space of constructible
functions on $E_{V, \Omega}$, such that
\[
E_u(CC(\IC(\line{S}, \C))) = (-1)^{{\rm dim} S} g_{S}.
\]
The above mentioned conjecture of Geiss-Leclerc-Schr{\"o}er can be made precise as follows.

\begin{conj}
\label{conj: multiplicity}
$Eu([T^{*}_{S}E_{V, \Omega}])= (-1)^{{\rm dim} S} \phi_{S}$ or equivalently $m_{S', S} = (-1)^{{\rm dim} S' - {\rm dim} S} n_{S', S}$.
\end{conj}

In this paper we develop a strategy to approach this conjecture. First let us formulate the dual statement. Let $M(E_{V, \Omega})^{G_V}$ be the space of $G_{V}$-invariant constructible functions on $E_{V, \Omega}$. Then Lusztig \cite{Lu91} showed that there is an algebra isomorphism
\[
U(\n) \cong \line{\mathcal{M}}_{\Omega} := \bigoplus_{V \in \mathcal{V}} M(E_{V, \Omega})^{G_V},
\]
where $\mathcal{V}$ is the set of isomorphism classes of $I$-graded vector spaces and the product on $\line{\mathcal{M}}_{\Omega}$ is given by convolution. So we can view the canonical and semi-canonical bases as elements in $\line{\mathcal{M}}_{\Omega}$. Let $M(\Lambda_{V})^{G_{V}}$ be the space of $G_{V}$-invariant constructible functions on $\Lambda_{V}$. The pullback along $E_{V, \Omega} \hookrightarrow \Lambda_{V}$ defines an algebra homomorphism
\[
\Psi: \line{\M}_{\Pi} := \bigoplus_{V \in \mathcal{V}} M(\Lambda_{V})^{G_{V}} \longrightarrow \line{\mathcal{M}}_{\Omega}
\]
where the product on $\line{\M}_{\Pi}$ is also given by convolution. Lusztig \cite{Lu91} showed that this induces an isomorphism $\Psi_{0}$ on a subalgebra $\M_{\Pi}$ of $\line{\M}_{\Pi}$. Let $\M_{\Pi}(V) := \M_{\Pi} \cap M(\Lambda_{V})^{G_V}$. We have a diagram
\begin{align}
\label{diagram}
\xymatrix{\M_{\Pi}(V) \ar@{^{(}->}[r] \ar[rd]^{\cong}_{\Psi_{0}} & M(\Lambda_{V})^{G_{V}} \ar[d]^{\Psi} \\
& M(E_{V, \Omega})^{G_{V}}.}
\end{align}
Lusztig \cite{lu00} showed that there exists a basis $\{\widetilde{\phi}_{S}\}$ of $\M_{\Pi}(V)$ parametrized by the $G_{V}$-orbits $S$ in $E_{V, \Omega}$ satisfying
\[
\widetilde{\phi}_{S}(x, y) = \begin{cases}
0 & \text{ if } (x, y) \in \mathcal{O}_{S'} \text{ and } S' \neq S \\
1 & \text{ if } (x, y) \in \mathcal{O}_{S} 
\end{cases}
\]
where $\mathcal{O}_{S}$ is some open dense subset of $T^*_{S}E_{V, \Omega}$. By definition, $\phi_S = \Psi_{0}(\widetilde{\phi}_{S})$. We define the dual semi-canonical basis to be $\rho_{S}(\phi) := \Psi_0^{-1}(\phi)$. Let $K_{G_{V}}(E_{V, \Omega})$ be the Grothendieck group of $G_{V}$-equivariant perverse sheaves on $E_{V, \Omega}$. The local Euler characteristic gives an isomorphism
\[
\chi: K_{G_{V}}(E_{V, \Omega}) \otimes_{\mathbb{Z}} \mathbb{C} \xrightarrow{\cong} M(E_{V, \Omega})^{G_V}, \quad \mathcal{F} \mapsto \phi_{\mathcal{F}}
\]
where $ \phi_{\mathcal{F}}(x) = \chi(\mathcal{F}_{x})$. Define $\chi^{{\rm mic}}_{S}(\phi_{\mathcal{F}}) := m_{S}(CC(\mathcal{F}))$, the multiplicity of $[T^*_{S}E_{V, \Omega}]$ in $CC(\mathcal{F})$. Then Conjecture~\ref{conj: multiplicity} is equivalent to the following dual statement.

\begin{conj}
\label{conj: dual}
\(
(-1)^{{\rm dim}\, S} \chi^{{\rm mic}}_{S} = \rho_S.
\)
\end{conj}

In order to approach this conjecture, we define a section of $\Psi_0$ 
\begin{align}
\label{eq: section}
\eta_{V}: M(E_{V, \Omega})^{G_V} \longrightarrow M(\Lambda_{V})^{G_V}
\end{align}
by 
\(
(\eta_{V}(\phi_{\mathcal{F}}))(x, y) := \chi({\rm R}\Phi_{f_{y}}[-1] (\mathcal{F})_{x}),
\)
where $f_{y}: E_{V, \Omega} \rightarrow \C$ is the linear functional defined by $y \in T^*_{S, x}E_{V, \Omega}$. This map has been introduced in \cite{CMMB} in a more general setting. The link with characteristic cycles is as follows.
\begin{prop}{\rm (cf.  Proposition \ref{prop: relation with CC})}
\label{prop: relation with CC 1}
For $\mathcal{F} \in D_{G_{V}}(E_{V, \Omega})$ and $(x, y) \in (T^*_{S}E_{V, \Omega})_{\rm reg}$, we have
\[
\eta_{V}(\phi_{\mathcal{F}})(x , y) = (-1)^{{\rm dim}\Lambda_{V} - {\rm dim} \widehat{S}}m_{S}(CC(\mathcal{F})).
\]
Here $\widehat{S}$ is the dual orbit of $S$.
\end{prop}
As a consequence, Conjecture~\ref{conj: dual} is equivalent to 
\[
\Psi_0^{-1}(\phi)|_{\mathcal{O}_{S}} = (-1)^{{\rm dim} \Lambda_{V} - {\rm dim} \widehat{S} - {\rm dim} S} \eta_{V}(\phi)|_{\mathcal{O}_{S}}
\]
for all $\phi \in M(E_{V, \Omega})^{G_V}$. Indeed, it is possible to show that
\begin{align}
\label{eq: sign}
{\rm dim} \Lambda_{V} - {\rm dim} \widehat{S} - {\rm dim} S \equiv 0 \text { mod } 2
\end{align}
from the fact that
\(
IC(\line{S}, \mathbb{C})^{\vee} = IC(\line{\widehat{S}}, \mathbb{C}),
\)
where $(\cdot)^{\vee}$ is the Fourier-Sato transform. We will not include the argument here, since it is not our main focus. Now we can state our main result.

\begin{teo}
\label{thm: inverse}
For type $A_2$ quiver,
\(
\Psi_0^{-1} = \eta_{V}.
\)
\end{teo}

Conjecture~\ref{conj: dual} for type $A_2$ quiver follows from this theorem and \eqref{eq: sign}. We shall point out that Conjecture~\ref{conj: dual} in this case also follows from the known results $(-1)^{{\rm dim} S} \chi^{{\rm mic}}_{S} = g^{*}_S$ \cite{Raicu:2016} and $\rho_{S} = g^{*}_{S}$ \cite{GLS05}, where $g^{*}_{S}$ is the dual canonical basis. Nevertheless, the purpose of this paper is to develop a strategy for studying Conjecture~\ref{conj: dual} in all cases. We plan to apply our strategy to some special orbits in the future.

The paper is organized as follows.  In \S \ref{section-characteristic-cycles} we review the notion of characteristic cycles. In \S \ref{section-canonical-basis} we review the classical work of 
Lusztig on the canonical bases and the semi-canonical bases. Both sections contain no new results and we mainly follow Lusztig's notations.  In \S \ref{subsection-Constructible-functions}, we introduce the map $\eta_V$ and show its image consists of constructible functions. The ideas are from \cite{CMMB}.
In \S \ref{subsection-characteristic-cycles}, we prove Proposition~\ref{prop: relation with CC}. Our main tool is stratified Morse theory (cf.
\cite{Schurmann:2003}). Note that in order to apply the results of \cite{Schurmann:2003}, some Whitney type regularity condition is required. This is verified in the appendix. In \S \ref{subsection-Compatibility}, we show that the equality $\eta_{V} = \Psi_0^{-1}$ is equivalent to the compatibility of $\eta_V$ with convolution (cf. Proposition \ref{prop-main-compatibility}). In \S \ref{subsection-Induction} and \S \ref{subsection-Further-reduction}, we reduce it further to a problem of vanishing cycle calculation. 

\begin{conj}
{\rm (cf. Conjecture\ref{conj: vanishing cycle})}
We have $\chi(R\Phi_{h_{y_0}}[-1](\mathbbm{1}))_{(1, x_0)} = 1$.
\end{conj}

Finally, in \S\ref{sec: A_2}, we show that the last conjecture is true for type $A_2$ quiver. We prove this result by showing that the relevant nearby cycle has Euler characteristic $0$. We should remark that even in this case, the singular locus of $h_{y_0}$ can be very complicated, and it could involve singular irreducible components of various dimensions. We show that the relevant nearby cycle has Euler characteristic $0$ by constructing a fibration of the Milnor fiber over some compact space and showing the fibers all have Euler characteristic $0$. Then the result follows from the Leray spectral sequence.

\par \vskip 1pc
{\bf Acknowledgement.}
The project was discussed when both authors were in Max Planck Institute for Mathematics of Bonn and started when both were in Yau Mathematical Sciences Center  of Tsinghua University. They would like to thank both institute for their excellent  working environment. The second author is supported by Tsinghua University Initiative Scientific Research Program No. 2019Z07L02016.

\section{Characteristic cycles}
\label{section-characteristic-cycles}

In this section we review some generalities on characteristic cycles, 
our main reference is \cite{ka13}. Nothing is new in this section.

\subsection{Micro-support and characteristic cycles}

To introduce the micro-support of a $\mC$-sheaf on a manifold, we 
follow \cite{ka13} section 8.6 to give a definition using vanishing cycles.

Consider a complex manifold $X$ with a holomorphic function 
\[
f: X\rightarrow \C.
\]
Moreover, we assume that $Y=f^{-1}(0)$ is non-singular.
Also, let $p:\C\rightarrow \C$ be the function $p(z)=\exp(2\pi \sqrt{-1}z)$, considered as the universal covering map of $\mC$.
Finally, let $\line{p}: \line{X}\rightarrow X$ be the pullback of 
$p$ along $f$. 

\begin{definition}
Let $\F\in D^{b}(X)$.  Let $i:Y\rightarrow X$ be the natural embedding. The nearby-cycle functor is defined by
\[
R\Psi_{f}(\F)=i^*R\line{p}_{*}\line{p}^*(\F).
\]
\end{definition}

We also need to consider the vanishing cycles, which is 

\begin{definition}
Let $R\Phi_f(\F)\in D^b(\F)$ be the unique element such that we have the following distinguished triangles
\[
i^*(\F)\rightarrow R\Psi_f(\F)\rightarrow R\Phi_f(\F) \longrightarrow^{\hspace{-0.5cm}+1} .
\]
\end{definition}

Now we can define the micro-support $SS(\F)$ of a constructible sheaf $\F$.
\begin{definition}
Let $D_{c}^{b}(X)$ be the subcategory of $D^b(X)$ consisting elements with bounded constructible cohomology sheaves. It is a full subcategory. Let $p\in T^*X$ and $\F\in D_c^{b}(X)$, then we define a subset $SS(\F)\subseteq T^*(X)$ by the following 
\begin{description}
\item[(1)]$p\notin SS(F)$.
\item[(2)]There exists an open neighborhood $U$ of $p$ such that for any $x\in X$ and any holomorphic function $f: W\rightarrow \C$ defined in
a neighborhood $W\subseteq X$ of $x$ with $f(x)=0$ and $df(x)\in U$, one gets 
$R\Phi_f(\F)_x=0$.
\end{description}

\end{definition}

\remk 
Note that such a definition works well for varieties over other fields. 
More precisely, Beilinson \cite{Be16}  constructed micro-support for 
arbitrary base field, and Saito \cite{Sa17} constructed characteristic cycle
for sheaves on varieties over a finite field.

\vspace{0.5cm}
Finally, following Kashiwara and Shapira, we can attach a Lagrangian cycle $CC(\F)$ to $\F\in D_{c}^{b}(X)$ in a functorial way.
Its support is $SS(\F)$. We call $CC(\F)$ the characteristic cycle of $\F$. We do not give the exact definition but just list some of its properties.

\begin{prop}\label{prop-basic-property-characteristic-cycle}
Let $X$ and $Y$ be complex manifolds, and $\F\in D_c^b(X), \G\in D_c^b(Y)$. We have 

\begin{description}
\item[(1)] $CC(\F\boxtimes \G)=CC(\F)\boxtimes CC(\G)$.
    \item[(2)] $CC(D_X(\F))=CC(\F)$, where $D_X$ is the Verdier dual.
    \item [(3)] Let $F'\rightarrow F\rightarrow F''\longrightarrow^{\hspace{-0.5cm}+1} $ be a distinguished triangle in $D_c^b(X)$. Then 
    \[
    CC(\F)=CC(\F')+CC(\F'').
    \]
\item[(4)]Assume that $\F$ is a local system. Then we have 
\[
CC(\F)=(-1)^{\dim(X)}\rank(\F)[T_X^*X].
\]
\item[(5)]We have $supp(CC(\F))=SS(\F)$.
\item[(6)](Milnor type formula) \, Let  $x\in U\subseteq X$ be an open subset. Suppose $f: U\rightarrow \C$ is holomorphic. Assume that the section $C_f=(y, df(y))$
of the natural projection $T^*X\rightarrow X$
 intersects $ SS(F)$ transversally. Then we have 
\[
- \chi(R\Phi_f(\F|_U)_x)=(CC(\F), C_f)_{T^*U, x}
\]
\item[(7)]Let $\F$ be perverse. Then 
\[
CC(\F)\geq 0.
\]
\end{description}

\end{prop}

\begin{proof}
(1) is (9.4.1) in \cite{ka13}, (2) is proved in Proposition
9.4.4 in loc.cit. and (3) is proved in Proposition 9.4.5 in loc.cit. 
Note that in (2) our formula differs from that of \cite{ka13}
by an antipodal pullback since we are working with complex  varieties. For (4), we refer to lemma 4.11 of \cite{Sa17}, and for (5) and (6), see Theorem 4.9 and Proposition 4.14 \cite{Sa17}. Again, we note that the characteristic cycle in \cite{ka13} differs from ours by a sign since we require $CC(\F)\geq 0$ for a perverse sheaf $\F$, following \cite{Sa17}.
Finally, $(7)$ follows from Proposition 5.14 of \cite{Sa17}. 
\end{proof}

\subsection{Constructible functions, Lagrangian cycles and characteristic cycles}

We introduce the following set of constructible functions on 
complex varieties.

\begin{definition}
A function $\phi: X\rightarrow \C$ is constructible if 
$f(X)$ is finite and $f^{-1}(c)$ is a constructible subset of $X$ in the Zarisky topology for any $c \in \mathbb{C}$. The set of constructible functions
on $X$ is denoted by $M(X)$.
\end{definition}
 
\remk
Our definition of constructible functions is more restrictive than that of \cite{ka13}. 
 
\begin{notation}
Let $K(X)$ be the Grothendieck group of $D_{c}^b(X)$, i.e,
the full subcategory of $D^b(X)$ consisting of elements  with bounded constructible 
 cohomology sheaves with $\C$-coefficients.
\end{notation}

\begin{definition}
Let $L(X)$ be the free abelian group generated by the
complex Lagrangian subvarieties of $X$. Here by Lagrangian subvariety we mean middle dimensional algebraic subvariety of $T^*X$.
\end{definition}
\remk 
Naturally we have $CC(\F)\in L(T^*X)$ for any $\F\in K(X)$.

\begin{definition}
We have group homomorphisms
\[
\chi: K(X) \rightarrow M(X), \quad \F\mapsto \chi(\F), \quad (\chi(\F))(x)=\chi(\F_x)
\]
and 
\[
CC: K(X)\rightarrow L(T^*X), \quad \F\mapsto CC(\F).
\]
\end{definition}

\begin{teo}
The homomorphisms $\chi$ and $CC$ are isomorphisms.
\end{teo}

\begin{proof}
Cf. \cite[Theorem 9.7.1, 9.7.10]{ka13}. Note that regardless of the modification we 
made on the relevant objects, the proof is exactly the same.
\end{proof}

Following \cite{ka13}, we define an Euler morphism $Eu$ from 
$L(T^*X)$ to $M(X)$ as follows 

\begin{definition}
Let $x\in X$, $U\subseteq X$ a neighborhood of $x$ and $\phi: U\rightarrow \R$ satisfying 
$\phi(x)=0, d\phi(x)=0$ and the Hessian of $\phi$ at $x$
is positive definite. Let $\lambda\in L(T^*X)$, then we put
\[
Eu(\lambda)(x)=\sharp([C_{\phi}]\cap \lambda)_{x}
\]
where $C_{\phi}=\{(y, d\phi(y))|y\in U\}$.
\end{definition}

\remk In loc.cit, it is shown to be well defined(cf. (9.7.26)). 

We are ready to state the following 

\begin{teo} \cite[Theorem 9.7.11]{ka13} \label{teo-Kashiwara-Shapira-constructible}
\label{teo-fundamental-commutative-diagram}
The diagram:
$$
\xymatrix{
 & K(X)\ar[dl]_{CC} \ar[dr]^{\chi}&\\
L(T^*X)\ar[rr]_{Eu}& &  M(X)
}
$$
is commutative, and the arrows are isomorphic.
\end{teo}

\section{Canonical Bases and Semi-canonical Bases}
\label{section-canonical-basis}

In this section we recall the classical construction of 
the canonical basis and the semi-canonical basis, due to Lusztig.
We only state the relevant facts in the case of  quivers of simply laced type, we refer to 
\cite{lu00}, \cite{Lu91} for a detailed discussion.

\subsection{Representation of quiver algebras and preprojective algebras}\label{section-quiver-definition}

Let $Q=(I, H, s, e)$ be a finite quiver without loops. Thus 
\begin{description}
\item[(1)] $I$ is a finite set of vertices;
\item[(2)] $H$ is a finite set of directed edges called arrows;
\item[(3)] $s(\text{ resp. }e): H\rightarrow I$ sends an arrow to its starting point (resp. end points);
\item[(4)]there is an involution $h\rightarrow \bar{h}$ satisfying $e(\bar{h})=s(h),s(\bar{h})=e(h)$.
\end{description}
Let $\Omega \subseteq H$ be an orientation, i.e., $\Omega \cup \line{\Omega} = H, \Omega \cap \line{\Omega} = \emptyset$. For $i\in I$, set 
\[
r_i=\sum_{h \in \Omega: s(h)=i} \line{h} h-\sum_{\alpha \in \Omega: e(h)=i} h\line{h}
\]

\begin{notation}
We denote by $(Q,\Omega)$  the sub-quiver generated by $\Omega$.
\end{notation}

\begin{definition}
Let $H(Q, \Omega)=\C \Omega$ be the quiver algebra generated by $\Omega$ and 
\[
\Pi(Q)=\C H/J
\]
where $J$ is the ideal generated by the elements $r_i$ above. We call $\Pi(Q)$ the 
preprojective algebra associated to $Q$.
\end{definition}

\begin{notation}
Let $V=\bigoplus_{i\in I}V_i$ be an $I$-graded vector space. Let
\[
|V| := (\dim(V_i))_{i\in I}
\]
be its dimension vector.
\end{notation}

\begin{definition}
View the variety 
\[
E_{V, \Omega}=\{(x_{h})_{h \in \Omega}: x_{h}\in \Hom(V_{s(h)}, V_{e(h)})\}
\]
 as the representation variety of $H(Q, \Omega)$ with underlying space $V$.
\end{definition}

\begin{definition}
A representation of the preprojective algebra $\Pi(Q)$ on $V$ is an element $(x_h)_{h \in H}\in E_{V, \Omega}\times E_{V, \line{\Omega}}$ satisfying the relation 
\[
\sum_{\{\alpha \in \Omega: s(h)=i\}}x_{\line{h}}x_{h}-\sum_{\{h \in \Omega: e(h)=i\}}x_{h}x_{\line{h}}=0.
\]
Let $p = h_1 h_2 \cdots h_t$ be a path in $H$. Set
\[
x_p = x_{h_1}x_{h_2}\cdots x_{h_t}
\]
We say that the representation is nilpotent if there is an $N$ in $\N$ such that $x_p=0$ for
for any path $p$ of length greater than  $N$.
Let $\Lambda_V$ be the set of nilpotent representations
on $V$.
\end{definition}

\remk
Note that if $V^{i}$ is an $I$-graded vector space with $|V^{i}| = (\delta_{i, j})_{j \in I}$ (here $\delta$ is the Kronecker symbol), then $\Lambda_{V^{i}}$ consists of one single point and we denote by $Z_i$ the corresponding representation. We also note that the nilpotency condition is equivalent to requiring that the representation admits a composition series consists of only simple modules isomorphic to $Z_i$ for $i \in I$.

\vspace{0.5cm}
We recall some basic results concerning the algebra $\Pi(Q)$.

\begin{prop}\cite[ Proposition 3.1]{GLS05}
The following are equivalent
\begin{itemize}
    \item [(a)] The algebra $\Pi(Q)$ is finite dimensional.
    \item [(b)] Every finite dimensional representation of 
    $\Pi(Q)$ is nilpotent.
    \item[(c)] $(Q,\Omega)$ is a Dynkin quiver.
\end{itemize}

\end{prop}

\subsection{Convolution products and canonical bases}
We recall the construction of canonical bases through
convolution products.

\begin{definition}
Let $X$ be a complex variety and $f: X\rightarrow \C$
a constructible function. We define
\[
\int_{x\in X} f(x)=\sum_{c \in \C}c \, \chi(f^{-1}(c)),
\]
where $\chi$ is the Euler characteristic with compact support.
\end{definition}

\begin{notation}
Let $G_V=\prod_{i\in I}GL(V_i)$ be the automorphism 
group of $V$, which acts on $E_{V,\Omega}$ and $\Lambda_V$
by conjugation.
\end{notation}

\begin{definition}
Let $M(E_{V,\Omega})^{G_{V}}$ be the set of $G_V$-invariant constructible functions on $E_V$. Similarly one can define
$M(\Lambda_V)^{G_V}$.

\end{definition}

\begin{definition}
Let $V, V', V''$ be $I$-graded vector spaces such that 
\[
|V| =|V'| + |V''|.
\]
Then we have a bilinear map 
\[
*: M(E_{V', \Omega})^{G_{V''}}\times M(E_{V'',\Omega})^{G_{V''}}\rightarrow M(E_{V,\Omega})^{G_V}
\]
by 
\[
(\phi'*\phi'')(x)=\int_{y\subseteq x}\phi'(y)\phi''(x/y), \quad x\in E_V
\]
where $y$ runs through all the subrepresentations of $x$
such that the underlying vector space is isomorphic to $V'$.
Similarly, we have a bilinear map 
\[
*: M(\Lambda_{V'})^{G_V}\times M(\Lambda_{V''})^{G_V}\rightarrow M(\Lambda_V)^{G_V}.
\]

\end{definition}

\begin{definition}
Let 
\[
\line{\M}_{\Omega}=\bigoplus_{V \in \mathcal{V}}M(E_{V, \Omega})^{G_V}, \quad \line{\M}_{\Pi}=\bigoplus_{V \in \mathcal{V}}M(\Lambda_V)^{G_V},
\]
where $\mathcal{V}$ is the set of isomorphism classes of $I$-graded vector spaces.

\end{definition}

\begin{prop}
The vector spaces $\line{\M}_\Omega$ and $\line{\M}_{\Pi}$ 
with the convolution product $*$ are unital associative algebras. 
\end{prop}

\begin{proof}
We refer to \cite{GLS05} section 5.4 and \cite{Lu91} section 10.19.
\end{proof}

\begin{definition}
Let $\M_\Omega$(resp. $\M_{\Pi}$) be the subalgebra of $\line{\M}_{\Omega}$(resp. $\line{\M}_{\Pi}$)
generated by the function $1_{S_i}$(resp. $1_{Z_i}$), $i\in I$, where $S_i$ (resp. $Z_i$) is the dimension 1 irreducible 
representation which is concentrated in degree $i$. Also, let
\[
\M_\Omega(V)=\M_\Omega\cap M(E_{V,\Omega})^{G_V}, \quad \M_{\Pi}(V)=\M_{\Pi} \cap M(\Lambda_V)^{G_V}
\]

\end{definition}

\begin{prop}\cite[Proposition 9.8]{Lu1}
\label{prop: simply laced}
If $Q$ is simply-laced, then $\M_\Omega = \line{\M}_{\Omega}$.
\end{prop}

\begin{notation}
Let $\mathfrak{g}$ be a symmetric Kac-Moody algebra and $\n$ a maximal nilpotent Lie subalgebra. Let $Q$ be the associated quiver. Also, let $U(\n)$ be the enveloping algebra of $\n$.
\end{notation}

\begin{teo}
We have isomorphisms of alebras 
\[
\Psi: U(\n)\rightarrow \M_\Omega, \quad \Phi: U(\n)\rightarrow \M_{\Pi}
\]
with
\[
\Psi(e_i)=1_{S_i}, \quad \Phi(e_i)=1_{Z_i},
\]
where $e_i, i\in I$ is a set of Chevalley generators for $U(\n)$.

\end{teo}

\begin{proof}
For $\Psi$, we refer to \cite{Lu91} Proposition 10.20, and for $\Phi$, we refer to 
\cite{lu00}.
\end{proof}

We give another description of the map $\Psi$ in terms 
of the quantum enveloping algebra, which is also due to Lusztig.
We briefly recall the construction. 

\begin{notation}
For each $V$, Lusztig defined a subset $\P_{\Omega}(V)$ of perverse sheaves on $E_{V, \Omega}$.
Let $\K_v(\Omega, V)$ be the $\Z[v^{\pm}]$-module generated by the elements of 
 $\P_{\Omega}(V)$. Moreover, he defined a convolution product
 \[
 *: \K_v(\Omega, V')\times \K_v(\Omega, V'')\rightarrow \K_v(\Omega, V)
 \]
 for $I$-graded vector spaces $V', V''$ such that $|V| = |V'| + |V''|$. Finally, let 
 \[
 \K_v(\Omega)=\bigoplus_{V \in \mathcal{V}} \K_v(\Omega, V)
 \] 
 be the resulting unital associative algebra.
\end{notation}

\begin{teo}
We have an isomorphism of algebras

\[
\Psi_v: U_v(\n)\rightarrow \K_v(\Omega) \otimes_{\mathbb{Z}[v^{\pm}]} \Q(v), \quad \Psi(E_i) = \mathbbm{1}_{S_i} 
\]
where $E_i, i\in I$ is a set of Chevalley generators for the quantized algebras $U_v(\n)$ and $\mathbbm{1}_{S_i}$ is the constant sheaf on the variety corresponding to the one dimensional representation $S_i$.
\end{teo}

\begin{proof}

Cf. \cite{Lu91}, \S 10.17.

\end{proof}

\remk 
By letting $v=1$, we recover the previous map $\Psi$ by identifying  $E_i$ to $e_i$ (cf. \cite{Lu91}, \S 10.20). 

\begin{definition}
Let $U_{v, \Z}(\n)=\Psi^{-1}_v(\K_v(\Omega))$.
\end{definition}

Let $\Irr(\Lambda_V)$ be the set of irreducible components of $\Lambda_V$.
\begin{definition}
Following Lusztig, we define for each graded vector space $V$ a $\C$-basis 
\[
\{\phi_Z | Z\in \Irr(\Lambda_V)\}
\]
of $\M_\Pi$. The function $\phi_Z$ is uniquely characterized by the fact that it is equal to 1 on a dense open subset of $Z$ and equal to 0 on a dense open subset of any other irreducible component $Z'$ of $\Lambda_V$ \cite[Lemma 2.5]{lu00}. 
\end{definition}

\section{Quiver of type A}

Let $Q = (I, H, s, e)$ be a quiver of type $A$. Let $I = \{1, 2, \cdots, r\}$ and $\Omega$ be the orientation $i \rightarrow i+1$. Let $V = \bigoplus_{i \in I} V_{i}$ be an $I$-graded vector space.
\[
E_{V, \Omega} = \bigoplus_{1 \leqslant i < r} {\rm Hom}(V_{i}, V_{i+1}), \quad G_{V} = \prod_{1 \leqslant i \leqslant r} GL(V_{i}).
\]
Let $D_{G_{V}}(E_{V, \Omega})$ be the $G_{V}$-equivariant derived category of constructible complexes on $E_{V, \Omega}$ and $K_{G_{V}}(E_{V, \Omega})$  the corresponding Grothendieck group. Then the local Euler characteristic gives an isomorphism
\[
\chi: K_{G_{V}}(E_{V, \Omega}) \otimes_{\mathbb{Z}} \mathbb{C} \xrightarrow{\sim} M(E_{V, \Omega})^{G_V}, \quad \mathcal{F} \mapsto \phi_{\mathcal{F}},
\]
where $\phi_{\mathcal{F}}(x) = \chi(\mathcal{F}_{x})$. Let $\bar{\Omega}$ be the opposite orientation, and 
\[
E_{V, \bar{\Omega}} = \bigoplus_{1 \leqslant i < r} {\rm Hom}(V_{i + 1}, V_{i}).
\]
Let 
\[
E_{V} = E_{V, \Omega} \oplus E_{V, \bar{\Omega}} \hookrightarrow {\rm End}(V), \quad G_{V} \hookrightarrow GL(V).
\]
We define a $GL(V)$-invariant nondegenerate bilinear form on ${\rm End}(V)$ by the trace
\[
\langle \,, \, \rangle: {\rm End}(V) \times {\rm End}(V) \rightarrow \mathbb{C}, \quad \langle x,  y \rangle = {\rm tr}(xy).
\]
It defines a $G_{V}$-invariant nondegenerate bilinear form on $\mathfrak{g}_{V} : = {\rm Lie}(G_{V})$, and a $G_{V}$-invariant nondegenerate pairing
\[
\langle \,, \, \rangle: E_{V, \Omega} \times E_{V, \bar{\Omega}} \rightarrow \mathbb{C}.
\]
Under this pairing, we can identify
\[
E_{V, \bar{\Omega}} \cong E_{V, \Omega}^{*} \text{ and } T^{*}E_{V, \Omega} \cong E_{V} \cong T^{*}E_{V, \bar{\Omega}}.
\]
On ${\rm End}(V)$, we have the Lie bracket $[x, y] := xy - yx$ and 
\[
\Lambda_{V} = \{(x, y) \in E_{V} \, | \, [x, y] = 0\}.
\]
We decompose $\Lambda_{V}$ into irreducible components, 
\[
\Lambda_{V} \cong \bigcup_{S} \line{T^*_{S}E_{V, \Omega}} \cong  \bigcup_{C} \line{T^*_{C}E_{V, \bar{\Omega}}}
\]
which are closures of conormal bundles over orbits on $E_{V, \Omega}$ and $E_{V, \bar{\Omega}}$ respectively. For any orbit $S \subseteq E_{V, \Omega}$, we define the dual orbit $\widehat{S} \subseteq E_{V, \bar{\Omega}}$ by the condition that
\[
\line{T^*_{S}E_{V, \Omega}} \cong \line{T^*_{\widehat{S}}E_{V, \bar{\Omega}}}.
\]  
We also define
\[
(T^*_{S}E_{V, \Omega})_{{\rm reg}} := T^*_{S}E_{V, \Omega} \backslash \cup_{S' \neq S} \line{T^*_{S'}E_{V, \Omega}}.
\]
Then it is easy to see that
\[
(T^*_{S}E_{V, \Omega})_{{\rm reg}} \subseteq S \times \widehat{S}.
\]

\subsection{Constructible functions}\label{subsection-Constructible-functions}

We will define a map
\[
\eta_{V}: M(E_{V, \Omega})^{G_V} \rightarrow M(\Lambda_V)^{G_V},
\]
which has been introduced in \cite{CMMB} in a more general setting.

\begin{definition}\label{definition-eta_V}
For any $(x, y) \in \Lambda_{V}$, $\eta_{V}(\phi_{\mathcal{F}})(x, y) = \chi(R\Phi_{f_{y}}[-1](\mathcal{F})_{x})$, where $f_y: E_{V,\Omega}\rightarrow \C$ is defined by
\(
f_y(z)=\langle \, z,y \, \rangle.
\)

\end{definition}

Next we show the image of $\eta_{V}$ lies in $M(\Lambda_V)^{G_V}$.

\begin{prop}
For $\mathcal{F} \in D_{G_{V}}(E_{V, \Omega})$, $\eta_{V}(\phi_{\mathcal{F}}) \in M(\Lambda_V)^{G_V}$.
\end{prop}

To prove this, we will give another description of $\eta_{V}$ following \cite{CMMB}. Let $S \subseteq E_{V, \Omega}$ be any orbit and $\widehat{S} \subseteq E_{V, \bar{\Omega}}$ be its dual. We would like to define $\eta_{V}(\phi_{\mathcal{F}})$ on each $T^{*}_{\widehat{S}}E_{V, \bar{\Omega}}$ as follows.
\[
\xymatrix{E_{V, \Omega} \times \widehat{S} \ar[r] \ar[d]_{\pi} & E_{V, \Omega} \times E_{V, \bar{\Omega}} \ar[d]^{{\rm \langle \, , \, \rangle}}   \\
E_{V, \Omega} & \mathbb{C}
}
\]
Note
\begin{align}
\label{eq: conormal}
T^{*}_{\widehat{S}}E_{V, \bar{\Omega}} \subseteq \langle \, , \, \rangle^{-1}(0).
\end{align}
Denote the restriction of $\langle \, , \, \rangle$ to $E_{V, \Omega} \times \widehat{S}$ by $f_{S}$.

\begin{lemma}
For $\mathcal{F} \in D_{G_{V}}(E_{V, \Omega})$ and $(x , y) \in T^{*}_{\widehat{S}}E_{V, \bar{\Omega}}$,
\begin{align}
\label{eq: lift vanishing cycle}
(R\Phi_{f_{S}}[-1] (\pi^{*}\mathcal{F}))_{(x, y)} \cong (R\Phi_{f_{y}}[-1] (\mathcal{F}))_{x}.
\end{align}
\end{lemma}

\begin{proof}
Let $Z_{G_{V}}(y)$ be the stabilizer of $y$ in $G_V$. We have an isomorphism
\[
G_{V} \times_{Z_{G_{V}}(y)} E_{V, \Omega} \cong E_{V, \Omega} \times \widehat{S}, \quad (g, z) \mapsto (gz, gy).
\]
The inclusion 
\begin{align}
\label{eq: descent}
E_{V, \Omega} \rightarrow G_{V} \times_{Z_{G_{V}}(y)} E_{V, \Omega}, \quad z \mapsto (1, z)
\end{align}
gives a section of $\pi$
\[
i: E_{V, \Omega} \rightarrow E_{V, \Omega} \times \widehat{S}, \quad z \mapsto (z, y).
\]
The pullback along the inclusion induces an equivalence of categories
\[
D_{G_{V}}(G_{V} \times_{Z_{G_{V}}(y)} E_{V, \Omega}) \cong D_{Z_{G_{V}}(y)}(E_{V, \Omega}).
\]
Since $\langle \, , \, \rangle$ is $G_{V}$-invariant, for any $\mathcal{G} \in D_{G_{V}}(E_{V, \Omega} \times \widehat{S}) \cong D_{G_{V}}(G_{V} \times_{Z_{G_{V}}(y)} E_{V, \Omega})$, we get
\[
R\Phi_{f_{y}}(i^{*}\mathcal{G}) \cong i^{*}R\Phi_{\langle \, , \, \rangle}(\mathcal{G}).
\]
Let $\mathcal{G} = \pi^{*}\mathcal{F}$, then $i^{*}\mathcal{G} = \mathcal{F}$. So
\[
R\Phi_{f_{y}}(\mathcal{F}) \cong i^{*}R\Phi_{f_{S}}(\pi^{*}\mathcal{F}).
\]
In particular,
\[
R\Phi_{f_{y}}(\mathcal{F})_{x} \cong i^{*}(R\Phi_{f_{S}}(\pi^{*}\mathcal{F}))_{x} \cong R\Phi_{f_{S}}(\pi^{*}\mathcal{F})_{(x, y)}.
\]
\end{proof}

\begin{rem}
In the lemma, we have used the following general fact. Suppose $H$ is a closed subgroup of $G$ and $X$ is an $H$-space. Then the inclusion 
\[
i: X \hookrightarrow G \times_{H} X, \quad x \mapsto (1, x)
\]
induces an equivalence of categories 
\[
D_{G}(G \times_{H} X) \cong D_{H}(X), \quad \mathcal{G} \mapsto i^{*}\mathcal{G}
\] 
Let $f: G \times_{H} X \rightarrow \mathbb{C}$ be a $G$-invariant continuous function. Then we have base change
\[
R\Phi_{f \circ i}(i^{*}\mathcal{G}) \cong i^{*}R\Phi_{f}(\mathcal{G}),
\]
for any $\mathcal{G} \in D_{G}(G \times_{H} X)$.
\end{rem}

\begin{cor}
$\eta_{V}(\phi_{\mathcal{F}})|_{T^{*}_{\widehat{S}}E_{V, \bar{\Omega}}} = \chi(R\Phi_{f_{S}}[-1] (\pi^{*}\mathcal{F}))|_{T^{*}_{\widehat{S}}E_{V, \bar{\Omega}}}$.
\end{cor}

In particular, $\eta_{V}(\phi_{\mathcal{F}})$ is constructible on $T^{*}_{\widehat{S}}E_{V, \bar{\Omega}}$. Since
\[
\Lambda_{V} = \bigsqcup_{S} T^{*}_{\widehat{S}} E_{V, \bar{\Omega}},
\]
we see $\eta_{V}(\phi_{\mathcal{F}}) \in M(\Lambda_V)^{G_V}$.

\subsection{Characteristic cycles}\label{subsection-characteristic-cycles}

For $\mathcal{F} \in D_{G_{V}}(E_{V, \Omega})$, let $m_{S}(CC(\mathcal{F}))$ be the multiplicity of $[T^*_{S}E_{V, \Omega}]$ in $CC(\mathcal{F})$. Let $f$ be the restriction of $\langle \, , \, \rangle$ to $S \times \widehat{S}$. The goal of this subsection is to prove the following proposition.
\begin{prop}
\label{prop: relation with CC}
For $\mathcal{F} \in D_{G_{V}}(E_{V, \Omega})$ and $(x, y) \in (T^{*}_{S}E_{V, \Omega})_{{\rm reg}}$,
\[
\eta_{V}(\phi_{\mathcal{F}})(x , y) = (-1)^{{\rm dim}\Lambda_{V} - {\rm dim} \widehat{S}}m_{S}(CC(\mathcal{F})).
\]
\end{prop}

The proof will occupy the whole section. Recall $\eta_{V}(\phi_{\mathcal{F}})(x , y) = \chi(R\Phi_{f_{y}}[-1](\mathcal{F})_{x})$. By \cite[Lemma 1.3.2]{Schurmann:2003},
\[
(R\Phi_{f_{y}}[-1](\mathcal{F}))_{x} \cong (R\Gamma_{re(f_{y}) \geqslant 0}(\mathcal{F}))_{x}
\]
In terms of stratified Morse theory, the right hand side is called the local Morse data, denoted by 
\(
{\rm LMD}(L, re(f_{y}), x).
\)
We have the following splitting formula for the local Morse data.

\begin{teo}
For $\mathcal{F} \in D_{G_{V}}(E_{V, \Omega})$ and $(x, y) \in (T^{*}_{S}E_{V, \Omega})_{{\rm reg}}$, 
\[
{\rm LMD}(\mathcal{F}, re(f_{y}), x) \cong {\rm TMD}(\mathcal{F}, re(f_{y}), x) \otimes^{L}_{\mathbb{C}} {\rm NMD}(\mathcal{F}, re(f_{y}), x)
\]
with
\begin{align*}
{\rm TMD}(\mathcal{F}, re(f_{y}), x) := (R\Gamma_{re(f_{y}) \geqslant 0}(\mathbbm{1}_{S}))_{x}  
\end{align*}
the tangential Morse data, and 
\begin{align*}
{\rm NMD}(\mathcal{F}, re(f_{y}), x) := (R\Gamma_{re(f_{y}) \geqslant 0}(\mathcal{F}|_{N_{S}}))_{x}   
\end{align*}
the normal Morse data with respect to a normal slice $N_{S} \subseteq E_{V, \Omega}$ to $S$ at $x$.
\end{teo}

\begin{proof}
It follows from \cite[Theorem 5.3.3]{Schurmann:2003}, which has some regularity condition on the stratification. We will verify this condition for our case in the appendix.
\end{proof}

As a direct consequence, we have

\begin{cor}
$\chi({\rm LMD}(\mathcal{F}, re(f_{y}), x)) = \chi({\rm TMD}(\mathcal{F}, re(f_{y}), x)) \cdot \chi({\rm NMD}(\mathcal{F}, re(f_{y}), x))$.
\end{cor}

It is the normal Morse data that relates to the characteristic cycle, namely
\[
(-1)^{{\rm dim} S}\chi({\rm NMD}(\mathcal{F}, re(f_{y}), x)) = m_{S}(CC(\mathcal{F})).
\]
This differs from \cite[(5.21)]{Schurmann:2003} by a sign $(-1)^{{\rm dim}\, S}$, which makes $m_{S}(CC(\mathcal{F}))$ positive whenever $\mathcal{F}$ is perverse. By \cite[Lemma1.3.2]{Schurmann:2003},
\begin{align*}
& {\rm TMD}(\mathcal{F}, re(f_{y}), x) \cong R\Phi_{f_{y}}[-1](\mathbbm{1}_{S})_{x} \\
& {\rm NMD}(\mathcal{F}, re(f_{y}), x) \cong R\Phi_{f_{y}}[-1](\mathcal{F}|_{N_{S}})_{x}. 
\end{align*}
So it remains to determine $\chi({\rm TMD}(\mathcal{F}, re(f_{y}), x))$. Instead of computing it directly, we shall apply the splitting formula to the other vanishing cycle $R\Phi_{f_{S}}[-1] (\pi^{*}\mathcal{F})_{(x, y)}$ in \eqref{eq: lift vanishing cycle}. By \cite[Lemma 1.3.2]{Schurmann:2003} again
\[
(R\Phi_{f_{S}}[-1] (\pi^{*}\mathcal{F}))_{(x, y)} \cong (R\Gamma_{re(f_{S}) \geqslant 0}(\pi^{*}\mathcal{F}))_{(x, y)} = {\rm LMD}(\pi^{*}\mathcal{F}, re(f_{S}), (x,y)).
\]
The stratification of $E_{V, \Omega}$ by $G_{V}$-orbits induces a stratification of $E_{V, \Omega} \times \widehat{S}$, which satisfies the same condition on regularity. Note $(x, y) \in S \times \widehat{S}$.

\begin{lemma}
$d(f_{S})|_{(x, y)} = (\pi^{*}df_{y})|_{(x, y)}$.
\end{lemma}

\begin{proof}
For $(v, w) \in T_{x}(E_{V, \Omega}) \oplus T_{y}(\widehat{S})$, let us choose curves $x(t), y(t)$ on $E_{V, \Omega}$ and $\widehat{S}$ respectively such that 
\begin{align*}
x(0) = x, x'(0) = v \text{ and } y(0) = y, y'(0) = w
\end{align*}
We compute the image of $(v, 0)$ and $(0, w)$ separately under $d(f_{S})_{(x, y)}$:
\begin{align*}
& (v, 0) \mapsto \frac{d \langle x(t), y \rangle}{t}|_{t = 0} = \langle v, y \rangle   \\
& (0, w) \mapsto \frac{d \langle x, y(t) \rangle}{t}|_{t = 0} = \langle x, w \rangle = 0
\end{align*}
where the last equality follows from \eqref{eq: conormal}. This finishes the proof.

\end{proof}

Since $(x, y) \in (T^{*}_{S}E_{V, \Omega})_{{\rm reg}}$, then $d(f_{S})|_{(x, y)} \in T^{*}_{S \times \widehat{S}}(E_{V, \Omega} \times \widehat{S})_{{\rm reg}}$. So we can apply the splitting formula again.

\begin{teo}
For $\mathcal{F} \in D_{G_{V}}(E_{V, \Omega})$ and $(x, y) \in (T^{*}_{S}E_{V, \Omega})_{{\rm reg}}$, 
\begin{align}
\label{eq: splitting formula}
{\rm LMD}(\pi^{*}\mathcal{F}, re(f_{S}), (x, y)) \cong {\rm TMD}(\pi^{*}\mathcal{F}, re(f_{S}), (x, y)) \otimes^{L}_{\mathbb{C}} {\rm NMD}(\pi^{*}\mathcal{F}, re(f_{S}), (x, y)).
\end{align}
where 
\begin{align*}
{\rm TMD}(\pi^{*}\mathcal{F}, re(f_{S}), (x, y)) & := (R\Gamma_{re (f) \geqslant 0} (1_{S \times \widehat{S}}))_{(x,y)}  \\ 
{\rm NMD}(\pi^{*}\mathcal{F}, re(f_{S}), (x, y)) & := (R\Gamma_{re (f_{S}|_{N_{S} \times \{y\}}) \geqslant 0} (\pi^{*}\mathcal{F}|_{N_{S} \times \{y\}}))_{(x, y)} 
\end{align*}
\end{teo}

\begin{proof}
It follows from \cite[Theorem 5.3.3]{Schurmann:2003}.
\end{proof}

By \cite[Lemma 1.3.2]{Schurmann:2003} again,
\begin{align*}
{\rm TMD}(\pi^{*}\mathcal{F}, re(f_{S}), (x, y)) & \cong R\Phi_{f}[-1](1_{S \times \widehat{S}})_{(x,y)} \\
{\rm NMD}(\pi^{*}\mathcal{F}, re(f_{S}), (x, y)) & \cong R\Phi_{f_{S}|_{N_{S} \times \{y\}}}[-1](\pi^{*}\mathcal{F}|_{N_{S} \times \{y\}})_{(x, y)}.
\end{align*}
By the natural isomorphism $N_{S} \cong N_{S} \times \{y\}$, we have
\[
R\Phi_{f_{S}|_{N_{S} \times \{y\}}}(\pi^{*}\mathcal{F}|_{N_{S} \times \{y\}})_{(x, y)} \cong R\Phi_{f_{y}|_{N_{S}}}(\mathcal{F}|_{N_{S}})_{x}.
\]
Hence,
\[
{\rm NMD}(\pi^{*}\mathcal{F}, re(f_{S}), (x, y)) \cong {\rm NMD}(\mathcal{F}, re(f_{y}), x).
\]
So it suffices to compute ${\rm TMD}(\pi^{*}\mathcal{F}, re(f_{S}), (x, y))$, equivalently $R\Phi_{f}[-1](\mathbbm{1}_{S \times \widehat{S}})_{(x,y)}$. 

\begin{prop}
\label{prop: tangential data}
$R\Phi_{f}[-1](\mathbbm{1}_{S \times \widehat{S}})_{(x,y)} = \mathbb{C}[{\rm dim}\Lambda_{V} - {\rm dim} \widehat{S} - {\rm dim} S]$.
\end{prop}

This proposition is a special case of \cite[Theorem 6.7.5]{CMMB}. For the convenience of the reader, we will reproduce its proof in the appendix.

\subsection{Compatibility with convolutions}\label{subsection-Compatibility}

In the introduction, we have considered the following diagram
\begin{align*}
\xymatrix{\M_{\Pi}(V) \ar@{^{(}->}[r] \ar[rd]^{\cong}_{\Psi_{0}} & M(\Lambda_{V})^{G_{V}} \ar[d]^{\Psi} \\
& M(E_{V, \Omega})^{G_{V}}}
\end{align*}
Starting from this subsection, we will investigate when $\eta_{V} = \Psi^{-1}_0$.

\begin{prop}\label{prop-main-compatibility}
The following statements are equivalent.

\begin{enumerate}

\item $\eta_{V} = \Psi^{-1}_0$;

\item ${\rm Im} \, \eta_{V} \subseteq \mathcal{M}_{\Pi}$;

\item For any decomposition of $I$-graded vector spaces $V = V^{1} \oplus V^{2}$,
\begin{align}
\label{eq: compatible with convolution}
\eta_{V^{1}}(\phi_{1}) \ast \eta_{V^{2}}(\phi_{2}) = \eta_{V}(\phi_{1} \ast \phi_{2})
\end{align}
for any $\phi_{1} \in M(E_{V^{1}, \Omega})^{G_{V^1}}$ and $\phi_{2} \in M(E_{V^{2}, \Omega})^{G_{V^2}}$.

\end{enumerate}

\end{prop}

\begin{proof}
Since $\Psi_0$ is an algebra isomorphism, then (1) implies (2) and (3). By the definition of $\eta_{V}$, we have $\Psi \circ \eta_{V} = id$. So (2) implies (1). It follows from (3) that
\[
\eta_{V}(1_{1} \ast \cdots \ast 1_{d}) = \eta_{V^{1}}(1_{1}) \ast \cdots \ast \eta_{V^{d}}(1_{d}) = 1_{\Lambda_{V^{1}}} \ast \cdots \ast 1_{\Lambda_{V^{d}}} \in \M_{\Pi}(V).
\]
(cf. \eqref{eq: main}).
By \cite[Proposition 7.3]{Lu91}, that $M(E_{V, \Omega})^{G_{V}}$ is spanned by $1_{1} \ast \cdots \ast 1_{d}$ for all $a \in S_{|V|}$ and associated decomposition of $I$-graded vector space $V = V^{1} \oplus \cdots \oplus V^{d}$. So (3) implies (2).
\end{proof}

We begin by recalling the definitions of the two convolutions in \eqref{eq: compatible with convolution}. Consider the following diagram
\begin{align}
\label{diag: convolution}
\xymatrix{
E'_{V^{1}, V^{2}, \Omega} \ar[r]^{p_{2}} \ar[d]_{p_{1}} & E''_{V^{1}, V^{2}, \Omega} \ar[d]^{p_{3}} \\
E_{V^{1}, \Omega} \times E_{V^{2}, \Omega} & E_{V, \Omega}
}
\end{align}
where
\begin{align*}
E''_{V^{1}, V^{2}, \Omega}  & := \Big\{(x, {\rm Fil}) | x \in E_{V, \Omega}, \, {\rm Fil}: 0 = W^{0} \subsetneq W^{1} \subsetneq W^{2} = V \text{ $x$-stable with }  | W^{k}/W^{k-1}| = |V^{k}| \text{ for } k = 1,2 \Big\} ,\\
E'_{V^{1}, V^{2}, \Omega}   & := \Big\{(x, {\rm Fil}, \varphi_{1}, \varphi_{2}) | (x, {\rm Fil}) \in E''_{V^{1}, V^{2}, \Omega}  \text{ and } \varphi_{k}: V^{k} \xrightarrow{\sim} W^{k}/W^{k-1} \text{ for } k = 1,2 \Big\} ,
\end{align*}
and
\[
p_{3}: E''_{V^{1}, V^{2}, \Omega} \rightarrow E_{V, \Omega}, \quad (x, {\rm Fil}) \mapsto x
\]
is proper;
\[
p_{2}: E'_{V^{1}, V^{2}, \Omega} \rightarrow E''_{V^{1}, V^{2}, \Omega}, \quad (x, {\rm Fil}, \varphi_{1}, \varphi_{2}) \mapsto (x, {\rm Fil})
\]
is a principal $G_{V^{1}} \times G_{V^{2}}$-bundle;
\[
p_{1}: E'_{V^{1}, V^{2}, \Omega} \rightarrow E_{V^{1}, \Omega} \times E_{V^{2}, \Omega}, \quad (x, {\rm Fil}, \varphi_{1}, \varphi_{2}) \mapsto (\varphi_{1}^{-1} x \varphi_{1}, \varphi_{2}^{-1} x \varphi_{2})
\]
is smooth, where we denote the induced morphisms on $W^{k}/W^{k-1}$ still by $x$. To see the properties of $p_{1}, p_{2}, p_{3}$ more easily, we will give another description of the diagram.

We fix a filtration $\overline{{\rm Fil}}: 0 = \bar{W}^{0} \subsetneq \bar{W}^{1} \subsetneq \bar{W}^{2} = V$, where $\bar{W}^{1} = V^{1}$. Let $\bar{\varphi}_{1}: V^{1} \rightarrow \bar{W}^{1}/\bar{W}^{0}$ be the identity and $\bar{\varphi}_{2}: V^{2} \rightarrow \bar{W}^{2}/\bar{W}^{1}$ be the composition of $V^{2} \rightarrow V \rightarrow V/V^{1}$. Let
\[
E_{V^{1}, V^{2}, \Omega}^{\geqslant 0} := \{x \in E_{V, \Omega} | x \text{ stabilizes } \overline{{\rm Fil}} \} \hookrightarrow E'_{V^{1}, V^{2}, \Omega}, \quad x \mapsto (x, \overline{{\rm Fil}}, \bar{\varphi}_{1}, \bar{\varphi}_{2})
\]
It admits an action by 
\[
G_{V^{1}, V^{2}}^{\geqslant 0} := \{g \in G_{V} | g \text{ stabilizes } \overline{{\rm Fil}} \}
\]
a parabolic subgroup of $G_{V}$. It has a Levi component $G_{V^{1}} \times G_{V^{2}}$ and the unipotent radical is
\[
G_{V^{1}, V^{2}}^{+} := \{g \in G_{V^{1}, V^{2}}^{\geqslant 0} | \bar{\varphi}_{k}^{-1} g \bar{\varphi}_{k} = id \text{ for } k = 1, 2\}.
\]
The following lemma is immediate.

\begin{lemma}
We have $G_{V}$-equivariant isomorphisms
\begin{align*}
G_{V} \times_{G_{V^{1}, V^{2}}^{\geqslant 0}} E_{V^{1}, V^{2}, \Omega}^{\geqslant 0} & \cong E''_{V^{1}, V^{2}, \Omega}, \quad (g, x) \mapsto (gx, g\overline{{\rm Fil}}) \\
G_{V} \times_{G_{V^{1}, V^{2}}^{+}} E_{V^{1}, V^{2}, \Omega}^{\geqslant 0} & \cong E'_{V^{1}, V^{2}, \Omega}, \quad (g, x) \mapsto (gx, g\overline{{\rm Fil}}, g\bar{\varphi}_{1}, g\bar{\varphi}_{2})
\end{align*}
\end{lemma}

By this lemma, we can rewrite diagram~\eqref{diag: convolution} as 
\begin{align}
\label{diag: convolution 1}
\xymatrix{
G_{V} \times_{G_{V^{1}, V^{2}}^{+}} E_{V^{1}, V^{2}, \Omega}^{\geqslant 0} \ar[r]^{p'_{2}} \ar[d]_{p'_{1}} & G_{V} \times_{G_{V^{1}, V^{2}}^{\geqslant 0}} E_{V^{1}, V^{2}, \Omega}^{\geqslant 0} \ar[d]^{p'_{3}} \\
E_{V^{1}, \Omega} \times E_{V^{2}, \Omega} & E_{V, \Omega}
}
\end{align}
where 
\begin{align*}
& p'_{3}: G_{V} \times_{G_{V^{1}, V^{2}}^{\geqslant 0}} E_{V^{1}, V^{2}, \Omega}^{\geqslant 0} \rightarrow E_{V, \Omega}, \quad (g, x) \mapsto gx \\
& p'_{2}:  G_{V} \times_{G_{V^{1}, V^{2}}^{+}} E_{V^{1}, V^{2}, \Omega}^{\geqslant 0} \rightarrow G_{V} \times_{G_{V^{1}, V^{2}}^{\geqslant 0}} E_{V^{1}, V^{2}, \Omega}^{\geqslant 0}, \quad (g, x) \mapsto (g, x) \\
& p'_{1}: G_{V} \times_{G_{V^{1}, V^{2}}^{+}} E_{V^{1}, V^{2}, \Omega}^{\geqslant 0} \rightarrow E_{V^{1}, V^{2}, \Omega}^{\geqslant 0} \hookrightarrow E'_{V^{1}, V^{2}, \Omega} \xrightarrow{p_{1}} E_{V^{1}, \Omega} \times E_{V^{2}, \Omega}
\end{align*}

For $\mathcal{F}_{1} \in D_{G_{V^{1}}}(E_{V^{1}, \Omega}), \mathcal{F}_{2} \in D_{G_{V^{2}}}(E_{V^{2}, \Omega})$, we define
\[
\mathcal{F}_{1} \ast \mathcal{F}_{2} := p_{3 !} \mathcal{F}''
\]
where 
\[
p_{2}^{*} \mathcal{F}'' \cong  p_{1}^{*} (\mathcal{F}_{1} \boxtimes \mathcal{F}_{2}).
\]
For $\phi_{1} \in M(E_{V^{1}, \Omega})^{G_{V^{1}}}, \phi_{2} \in  M(E_{V^{1}, \Omega})^{G_{V^{1}}}$, we define $\phi_{1} \ast \phi_{2} \in M(E_{V, \Omega})^{G_V}$ by
\[
(\phi_{1} \ast \phi_{2}) (x) = \int_{p_{3}^{-1}(x)} \phi''(x, {\rm Fil})
\]
where 
\[
\phi''(x, {\rm Fil}) = \phi_{1}(x_{1}) \phi_{2}(x_{2})
\]
with 
\[
x_{1} = \varphi_{1}^{-1} x \varphi_{1}, x_{2} = \varphi_{2}^{-1} x \varphi_{2}
\]
for any choice of isomorphisms $\varphi_{k}: V^{k} \rightarrow W^{k}/W^{k-1}$. 

\begin{prop}
For $\mathcal{F}_{1} \in D_{G_{V^{1}}}(E_{V^{1}, \Omega}), \mathcal{F}_{2} \in D_{G_{V^{2}}}(E_{V^{2}, \Omega})$,
\[
\phi_{\mathcal{F}_{1}} \ast \phi_{\mathcal{F}_{2}} = \phi_{\mathcal{F}_{1} \ast \mathcal{F}_{2}}.
\]
\end{prop}

\begin{proof}
We have
\[
\phi_{\mathcal{F}_{1} \ast \mathcal{F}_{2}}(x) = \chi( H^{*}(p_{3}^{-1}(x), \mathcal{F}'') ) = \int_{p_{3}^{-1}(x)} \chi(\mathcal{F}''),
\]
where for the second equality we refer to \cite{Vogan} Proposition 24.16.
It is easy to see that $\chi(\mathcal{F}'') = \phi''$.
\end{proof}

Next consider the following diagram

\begin{align}
\label{diag: convolution conormal}
\xymatrix{
\Lambda'_{V^{1}, V^{2}} \ar[r]^{q_{2}} \ar[d]_{q_{1}} & \Lambda''_{V^{1}, V^{2}} \ar[d]^{q_{3}} \\
\Lambda_{V^{1}} \times \Lambda_{V^{2}} & \Lambda_{V}
}
\end{align}
where
\begin{align*}
\Lambda''_{V^{1}, V^{2}}  & := \Big\{(x, y, {\rm Fil}) | (x, y) \in \Lambda_{V}, \, {\rm Fil}: 0 = W^{0} \subsetneq W^{1} \subsetneq W^{2} = V \text{ $(x, y)$-stable with }  | W^{k}/W^{k-1}| = |V^{k}| \text{ for } k = 1,2 \Big\} \\
\Lambda'_{V^{1}, V^{2}}   & := \Big\{(x, y, {\rm Fil}, \varphi_{1}, \varphi_{2}) | (x, y, {\rm Fil}) \in \Lambda''_{V^{1}, V^{2}}  \text{ and } \varphi_{k}: V^{k} \xrightarrow{\sim} W^{k}/W^{k-1} \text{ for } k = 1,2 \Big\} ,
\end{align*}
and
\[
q_{3}: \Lambda''_{V^{1}, V^{2}} \rightarrow \Lambda_{V}, \quad (x, y, {\rm Fil}) \mapsto (x, y)
\]
is proper;
\[
q_{2}: \Lambda'_{V^{1}, V^{2}} \rightarrow \Lambda''_{V^{1}, V^{2}}, \quad (x, y, {\rm Fil}, \varphi_{1}, \varphi_{2}) \mapsto (x, y, {\rm Fil})
\]
is a principal $G_{V^{1}} \times G_{V^{2}}$-bundle;
\[
q_{1}: \Lambda'_{V^{1}, V^{2}} \rightarrow \Lambda_{V^{1}} \times \Lambda_{V^{2}}, \quad (x, y, {\rm Fil}, \varphi_{1}, \varphi_{2}) \mapsto ((\varphi_{1}^{-1} x \varphi_{1}, \varphi_{1}^{-1} y \varphi_{1}), (\varphi_{2}^{-1} x \varphi_{2}, \varphi_{2}^{-1} y \varphi_{2}))
\]
where we denote the induced morphisms on $W^{k}/W^{k-1}$ still by $x, y$. 

For $\phi_{1} \in M(\Lambda_{V^{1}})^{G_{V^{1}}}, \phi_{2} \in  M(\Lambda_{V^{2}})^{G_{V^{2}}}$, we define $\phi_{1} \ast \phi_{2} \in M(\Lambda_{V})^{G_{V}}$ by
\[
(\phi_{1} \ast \phi_{2}) (x, y) = \int_{q_{3}^{-1}(x, y)} \phi''(x, y, {\rm Fil})
\]
where 
\[
\phi''(x, y, {\rm Fil}) = \phi_{1}((x_{1}, y_{1})) \phi_{2}((x_{2}, y_{2}))
\]
with 
\[
x_{1} = \varphi_{1}^{-1} x \varphi_{1}, \, y_{1} = \varphi_{1}^{-1} y \varphi_{1}; \quad x_{2} = \varphi_{2}^{-1} x \varphi_{2}, \, y_{2} = \varphi_{2}^{-1} y \varphi_{2}
\]
for any choice of isomorphisms $\varphi_{k}: V^{k} \rightarrow W^{k}/W^{k-1}$. 

One can easily extend this convolution to constructible functions on $E_{V^{1}}, E_{V^{2}}$ by considering the diagram
\[
\xymatrix{
& E'_{V^{1}, V^{2}} \ar[rr]^{q'_{2}} \ar[dd]_{q'_{1}} && E''_{V^{1}, V^{2}} \ar[dd]^{q'_{3}} \\
\Lambda'_{V^{1}, V^{2}} \ar[ur] \ar[rr]^{\quad \quad \quad q_{2}} \ar[dd]_{q_{1}} && \Lambda''_{V^{1}, V^{2}} \ar[ur] \ar[dd]^{q_{3}} & \\
& E_{V^{1}} \times E_{V^{2}} && E_{V} \\
\Lambda_{V^{1}} \times \Lambda_{V^{2}} \ar[ur] && \Lambda_{V} \ar[ur] & 
}
\]
where
\begin{align*}
E''_{V^{1}, V^{2}}  & := \Big\{(x, y, {\rm Fil}) | (x, y) \in E_{V}, \, {\rm Fil}: 0 = W^{0} \subsetneq W^{1} \subsetneq W^{2} = V \text{ $(x, y)$-stable with }  | W^{k}/W^{k-1}| = |V^{k}| \text{ for } k = 1,2 \Big\} ,\\
E'_{V^{1}, V^{2}}   & := \Big\{(x, y, {\rm Fil}, \varphi_{1}, \varphi_{2}) | (x, y, {\rm Fil}) \in E''_{V^{1}, V^{2}}  \text{ and } \varphi_{k}: V^{k} \xrightarrow{\sim} W^{k}/W^{k-1} \text{ for } k = 1,2 \Big\} .
\end{align*}
Note both the top and right squares are Cartesian, but the left one is not. The following lemma is immediate from the definition.

\begin{lemma}
For $\phi_{1} \in M(E_{V^{1}})^{G_{V^{1}}}, \phi_{2} \in M(E_{V^{2}})^{G_{V^{2}}}$, 
\[
(\phi_{1} \ast \phi_{2})|_{\Lambda_{V}} = \phi_{1}|_{\Lambda_{V^{1}}} \ast \phi_{2}|_{\Lambda_{V^{2}}}.
\]
\end{lemma}

For $\mathcal{F}_{1} \in D_{G_{V^{1}}}(E_{V^{1}}), \mathcal{F}_{2} \in D_{G_{V^{2}}}(E_{V^{2}})$, we can also define convolution 
\[
\mathcal{F}_{1} \ast \mathcal{F}_{2} := q'_{3 !} \mathcal{F}''
\]
where 
\[
q_{2}^{'*} \mathcal{F}'' \cong  q_{1}^{'*} (\mathcal{F}_{1} \boxtimes \mathcal{F}_{2}).
\]
One can also show easily that
\[
\phi_{\mathcal{F}_{1}} \ast \phi_{\mathcal{F}_{2}} = \phi_{\mathcal{F}_{1} \ast \mathcal{F}_{2}}.
\]

Now we attempt to establish \eqref{eq: compatible with convolution} directly. It suffices to show that for any $\mathcal{F}_{1} \in D_{G_{V^{1}}}(E_{V^{1}, \Omega}), \mathcal{F}_{2} \in D_{G_{V^{2}}}(E_{V^{2}, \Omega})$, 
\[
\eta_{V^{1}}(\phi_{\mathcal{F}_{1}}) \ast \eta_{V^{2}}(\phi_{\mathcal{F}_{2}}) = \eta_{V}(\phi_{\mathcal{F}_{1} \ast \mathcal{F}_{2}}).
\]
We first expand the right hand side. For $(x, y) \in \Lambda_{V}$, 
\[
\eta_{V}(\phi_{\mathcal{F}_{1} \ast \mathcal{F}_{2}})(x, y) = \chi(R\Phi_{f_{y}}[-1](\mathcal{F}_{1} \ast \mathcal{F}_{2})_{x}) = \chi(R\Phi_{f_{y}}[-1](p_{3!} \mathcal{F}'')_{x}).
\]
By proper base change,
\[
R\Phi_{f_{y}}(p_{3!} \mathcal{F}'') \cong p_{3!}(R\Phi_{f_{y} \circ p_{3}}\mathcal{F}'').
\]
Hence,
\[
\chi(R\Phi_{f_{y}}(p_{3!} \mathcal{F}'')_{x}) = \chi(H^{*}(p_{3}^{-1}(x), R\Phi_{f_{y} \circ p_{3}}\mathcal{F}'') ) = \int_{p_{3}^{-1}(x)} \chi(R\Phi_{f_{y} \circ p_{3}}\mathcal{F}'')
\]
where for the second equality we refer to \cite{Vogan} Proposition 24.16.
By smooth base change,
\[
p_{2}^{*}(R\Phi_{f_{y} \circ p_{3}}\mathcal{F}'') \cong R\Phi_{f_{y} \circ p_{3} \circ p_{2}} (p_{2}^{*}\mathcal{F}'') .
\]
We can also express the left hand side as an integration.
\[
(\eta_{V^{1}}(\phi_{\mathcal{F}_{1}}) \ast \eta_{V^{2}}(\phi_{\mathcal{F}_{2}}))(x, y) = \int_{q_{3}^{-1}(x, y)} \eta_{V^{1}}(\phi_{\mathcal{F}_{1}})(x_{1}, y_{1}) \,\cdot \, \eta_{V^{2}}(\phi_{\mathcal{F}_{2}})(x_{2}, y_{2})
\]
Comparing the two integrals, we see
\[
q_{3}^{-1}(x, y) \hookrightarrow p_{3}^{-1}(x), \quad (x, y, {\rm Fil}) \mapsto (x, {\rm Fil}).
\]
Then both integrals are equal if 
\begin{enumerate}

\item $\chi(R\Phi_{f_{y} \circ p_{3}}\mathcal{F}'') = 0$ over $p_{3}^{-1}(x) \backslash q_{3}^{-1}(x, y)$.

\item $\chi(R\Phi_{f_{y} \circ p_{3}} [-1]\mathcal{F}'')(x, {\rm Fil}) = \eta_{V^{1}}(\phi_{\mathcal{F}_{1}})(x_{1}, y_{1}) \cdot  \eta_{V^{2}}(\phi_{\mathcal{F}_{2}})(x_{2}, y_{2})$ for $(x, y, {\rm Fil}) \in q_{3}^{-1}(x, y)$.

\end{enumerate}
In (2), we can rewrite the right hand side as 
\[
\eta_{V^{1}}(\phi_{\mathcal{F}_{1}})(x_{1}, y_{1}) \cdot \eta_{V^{2}}(\phi_{\mathcal{F}_{2}})(x_{2}, y_{2}) = \chi(R\Phi_{f_{y_{1}}}[-1] (\mathcal{F}_{1})_{x_{1}}) \,\chi(R\Phi_{f_{y_{2}}}[-1] (\mathcal{F}_{2})_{x_{2}}).
\]

\begin{teo}[Sebastiani-Thom]
\[
R\Phi_{f_{y_{1}}}[-1] (\mathcal{F}_{1})_{x_{1}} \otimes^{L}_{\mathbb{C}} R\Phi_{f_{y_{2}}}[-1] (\mathcal{F}_{2})_{x_{2}} \cong R\Phi_{f_{y_{1}} \oplus f_{y_{2}}}[-1] (\mathcal{F}_{1} \boxtimes \mathcal{F}_{2})_{(x_{1}, x_{2})},
\] 
where
\[
f_{y_{1}} \oplus f_{y_{2}} : E_{V^{1}, \Omega} \times E_{V^{2}, \Omega} \rightarrow \mathbb{C}, \quad (x'_{1}, x'_{2}) \mapsto f_{y_{1}}(x'_{1}) + f_{y_{2}}(x'_{2}).
\]
\end{teo}

By smooth base change, 
\[
p_{1}^{*} R\Phi_{f_{y_{1}} \oplus f_{y_{2}}} (\mathcal{F}_{1} \boxtimes \mathcal{F}_{2}) \cong R\Phi_{(f_{y_{1}} \oplus f_{y_{2}}) \circ p_{1}} (p_{1}^{*}(\mathcal{F}_{1} \boxtimes \mathcal{F}_{2})).
\]
So (2) is equivalent to
\[
\chi(R\Phi_{(f_{y_{1}} \oplus f_{y_{2}}) \circ p_{1}} (p_{1}^{*}(\mathcal{F}_{1} \boxtimes \mathcal{F}_{2}))_{(x, y, {\rm Fil})}) = \chi(R\Phi_{f_{y} \circ p_{3} \circ p_{2}} (p_{2}^{*}\mathcal{F}'')_{(x, y, {\rm Fil})}). 
\]
Note $p_{1}^{*}(\mathcal{F}_{1} \boxtimes \mathcal{F}_{2}) \cong p_{2}^{*}\mathcal{F}''$, but 
\[
(f_{y_{1}} \oplus f_{y_{2}}) \circ p_{1} \neq f_{y} \circ p_{3} \circ p_{2}.
\]
So we can not conclude the equality directly. This is the main reason that we have to approach \eqref{eq: compatible with convolution} in a roundabout way.

\subsection{Inductions}\label{subsection-Induction}

We fix an $I$-graded isomorphism $V = V^{1} \oplus V^{2} \oplus \cdots \oplus V^{n}$ and a filtration
\[
\overline{{\rm Fil}}: 0 = \bar{W}^{0} \subsetneq \bar{W}^{1} \subsetneq \cdots \subsetneq \bar{W}^{n} = V
\]
where 
\[
\bar{W}^{k} := V^{1} \oplus \cdots \oplus V^{k}
\]
Let $\bar{\varphi}_{k}: V^{k} \hookrightarrow \bar{W}^{k} \mapsto \bar{W}^{k}/\bar{W}^{k-1}$. The goal is to calculate 
\[
\phi_{1} \ast \cdots \ast \phi_{n}
\]
for $\phi_{k} \in M(E_{V^{k}, \Omega})^{G_{V^{k}}}$. We will define
\[
{\rm Ind}_{V^{1}, \cdots, V^{n}}: \otimes_{k = 1}^{n}M(E_{V^{k}, \Omega})^{G_{V^{k}}} \rightarrow M(E_{V, \Omega})^{G_V}
\]
as follows. Consider the diagram
\begin{align}
\label{diag: induction}
\xymatrix{
E'_{V^{1}, \cdots, V^{n}, \Omega} \ar[r]^{p_{2}} \ar[d]_{p_{1}} & E''_{V^{1}, \cdots, V^{n}, \Omega} \ar[d]^{p_{3}} \\
E_{V^{1}, \Omega} \times \cdots \times E_{V^{n}, \Omega} & E_{V, \Omega}
}
\end{align}
where
\begin{align*}
E''_{V^{1}, \cdots V^{n}, \Omega}  & := \Big\{(x, {\rm Fil}) | x \in E_{V, \Omega}, \, {\rm Fil}: 0 = W^{0} \subsetneq \cdots \subsetneq W^{n} = V \text{ $x$-stable with }  | W^{k}/W^{k-1}| = |V^{k}| \text{ for } k = 1, \cdots, n \Big\}, \\
E'_{V^{1}, \cdots V^{n}, \Omega}   & := \Big\{(x, {\rm Fil}, \{\varphi_{k}\}_{k=1}^{n}) | (x, {\rm Fil}) \in E''_{V^{1}, \cdots V^{n}, \Omega}  \text{ and } \varphi_{k}: V^{k} \xrightarrow{\sim} W^{k}/W^{k-1} \text{ for } k = 1, \cdots, n \Big\} ,
\end{align*}
and
\[
p_{3}: E''_{V^{1}, \cdots, V^{n}, \Omega} \rightarrow E_{V, \Omega}, \quad (x, {\rm Fil}) \mapsto x
\]
is proper;
\[
p_{2}: E'_{V^{1}, \cdots, V^{n}, \Omega} \rightarrow E''_{V^{1}, \cdots, V^{n}, \Omega}, \quad (x, {\rm Fil}, \{\varphi_{k}\}_{k=1}^{n}) \mapsto (x, {\rm Fil})
\]
is a principal $G_{V^{1}} \times \cdots \times G_{V^{n}}$-bundle;
\[
p_{1}: E'_{V^{1}, \cdots, V^{n}, \Omega} \rightarrow E_{V^{1}, \Omega} \times \cdots \times E_{V^{n}, \Omega}, \quad (x, {\rm Fil}, \{\varphi_{k}\}_{k=1}^{n}) \mapsto \{\varphi_{k}^{-1}x\varphi_{k}\}_{k=1}^{n}
\]
is smooth, where we denote the induced morphisms on $W^{k}/W^{k-1}$ still by $x$. To see the properties of $p_{1}, p_{2}, p_{3}$ more easily, we will give another description of the diagram. Let
\[
E_{V^{1}, \cdots, V^{n}, \Omega}^{\geqslant 0} := \{x \in E_{V, \Omega} | x \text{ stabilizes } \overline{{\rm Fil}} \} \hookrightarrow E'_{V^{1}, \cdots, V^{n}, \Omega}, \quad x \mapsto (x, \overline{{\rm Fil}}, \{\bar{\varphi}_{k}\}_{k=1}^{n})
\]
It admits an action by 
\[
G_{V^{1}, \cdots, V^{n}}^{\geqslant 0} := \{g \in G_{V} | g \text{ stabilizes } \overline{{\rm Fil}} \}
\]
a parabolic subgroup of $G_{V}$. It has a Levi component $G_{V^{1}} \times \cdots \times G_{V^{n}}$ and the unipotent radical is
\[
G_{V^{1}, \cdots, V^{n}}^{+} := \{g \in G_{V^{1}, \cdots, V^{n}}^{\geqslant 0} | \bar{\varphi}_{k}^{-1} g \bar{\varphi}_{k} = id \text{ for } k = 1, \cdots, n\}.
\]

\begin{lemma}
\begin{align*}
G_{V} \times_{G_{V^{1}, \cdots, V^{n}}^{\geqslant 0}} E_{V^{1}, \cdots, V^{n}, \Omega}^{\geqslant 0} & \cong E''_{V^{1}, \cdots, V^{n}, \Omega}, \quad (g, x) \mapsto (gx, g\overline{{\rm Fil}}) \\
G_{V} \times_{G_{V^{1}, \cdots, V^{n}}^{+}} E_{V^{1}, \cdots, V^{n}, \Omega}^{\geqslant 0} & \cong E'_{V^{1}, \cdots, V^{n}, \Omega}, \quad (g, x) \mapsto (gx, g\overline{{\rm Fil}}, \{g\bar{\varphi}_{k}\}_{k = 1}^{n})
\end{align*}
\end{lemma}

\begin{align}
\label{diag: induction 1}
\xymatrix{
G_{V} \times_{G_{V^{1}, \cdots, V^{n}}^{+}} E_{V^{1}, \cdots, V^{n}, \Omega}^{\geqslant 0} \ar[r]^{p'_{2}} \ar[d]_{p'_{1}} & G_{V} \times_{G_{V^{1}, \cdots, V^{n}}^{\geqslant 0}} E_{V^{1}, \cdots, V^{n}, \Omega}^{\geqslant 0} \ar[d]^{p'_{3}} \\
E_{V^{1}, \Omega} \times \cdots \times E_{V^{n}, \Omega} & E_{V, \Omega}
}
\end{align}
where 
\begin{align*}
& p'_{3}: G_{V} \times_{G_{V^{1}, \cdots, V^{n}}^{\geqslant 0}} E_{V^{1}, \cdots, V^{n}, \Omega}^{\geqslant 0} \rightarrow E_{V, \Omega}, \quad (g, x) \mapsto gx, \\
& p'_{2}:  G_{V} \times_{G_{V^{1}, \cdots, V^{n}}^{+}} E_{V^{1}, \cdots, V^{n}, \Omega}^{\geqslant 0} \rightarrow G_{V} \times_{G_{V^{1}, \cdots, V^{n}}^{\geqslant 0}} E_{V^{1}, \cdots, V^{n}, \Omega}^{\geqslant 0}, \quad (g, x) \mapsto (g, x), \\
& p'_{1}: G_{V} \times_{G_{V^{1}, \cdots, V^{n}}^{+}} E_{V^{1}, \cdots, V^{n}, \Omega}^{\geqslant 0} \rightarrow E_{V^{1}, \cdots, V^{n}, \Omega}^{\geqslant 0} \hookrightarrow E'_{V^{1}, \cdots, V^{n}, \Omega} \xrightarrow{p_{1}} E_{V^{1}, \Omega} \times \cdots \times E_{V^{n}, \Omega}.
\end{align*}

For $\phi_{k} \in M(E_{V^{k}, \Omega})^{G_{V^{k}}} (k =1, \cdots, n)$, we define 
\[
{\rm Ind}_{V^{1}, \cdots, V^{n}}(\phi_{1} \otimes \cdots \otimes \phi_{n}) (x) = \int_{p_{3}^{-1}(x)} \phi''(x, {\rm Fil})
\]
where 
\[
\phi''(x, {\rm Fil}) = \phi_{1}(x_{1}) \cdots \phi_{n}(x_{n})
\]
with 
\[
x_{k} = \varphi_{k}^{-1} x \varphi_{k}
\]
for any choice of isomorphisms $\varphi_{k}: V^{k} \rightarrow W^{k}/W^{k-1}$. For $\mathcal{F}_{k} \in D_{G_{V^{k}}}(E_{V^{k}, \Omega}) (k =1, \cdots, n)$, we define
\[
{\rm Ind}_{V^{1}, \cdots, V^{n}} (\mathcal{F}_{1} \boxtimes \cdots \boxtimes \mathcal{F}_{n}) := p_{3 !} \mathcal{F}''
\]
where 
\[
p_{2}^{*} \mathcal{F}'' \cong  p_{1}^{*} (\mathcal{F}_{1} \boxtimes \cdots \boxtimes \mathcal{F}_{n}).
\]

\begin{prop}
For $\mathcal{F}_{k} \in D_{G_{V^{k}}}(E_{V^{k}, \Omega}) (k = 1, \cdots, n)$,
\[
{\rm Ind}_{V^{1}, \cdots, V^{n}} (\phi_{\mathcal{F}_{1}} \otimes \cdots \otimes \phi_{\mathcal{F}_{n}}) = \phi_{{\rm Ind}_{V^{1}, \cdots, V^{n}}(\mathcal{F}_{1} \boxtimes \cdots \boxtimes \mathcal{F}_{n})}.
\]
\end{prop}

\begin{proof}We have
\[
\phi_{{\rm Ind}_{V^{1}, \cdots, V^{n}}(\mathcal{F}_{1} \boxtimes \cdots \boxtimes \mathcal{F}_{n})}(x) = \chi( H^{*}(p_{3}^{-1}(x), \mathcal{F}'') ) = \int_{p_{3}^{-1}(x)} \chi(\mathcal{F}'')
\]
where for the second equality we refer to \cite{Vogan} Proposition 24.16. It is easy to see that $\chi(\mathcal{F}'') = \phi''$.
\end{proof}

\begin{prop}
\label{prop: convolution-induction}
For $\phi_{k} \in M(E_{V^{k}, \Omega})^{G_{V^{k}}} (k =1, \cdots, n)$, 
\[
\phi_{1} \ast \cdots \ast \phi_{n} = {\rm Ind}_{V^{1}, \cdots, V^{n}}(\phi_{1} \otimes \cdots \otimes \phi_{n})
\]
\end{prop}

\begin{proof}
We will prove it by induction on $n$. When $n = 2$, there is nothing to show. Suppose $n > 2$. By induction assumption,
\begin{align*}
\phi_{1} \ast \cdots \ast \phi_{n} & = (\phi_{1} \ast \phi_{2}) \ast \cdots \ast \phi_{n} = {\rm Ind}_{V^{1}, V^{2}} (\phi_{1} \otimes \phi_{2}) \ast \cdots \ast \phi_{n} \\
& = {\rm Ind}_{\bar{W}^{2}, \cdots, V^{n}}({\rm Ind}_{V^{1}, V^{2}}(\phi_{1} \otimes \phi_{2}) \otimes \cdots \otimes \phi_{n})
\end{align*}
Consider the following diagram
\[
\xymatrix{
&&& E'_{V^{1}, \cdots, V^{n}, \Omega}  \ar[rd] \ar[llldddd]&&
\\
& \tilde{E}'_{V^{1}, \cdots, V^{n}, \Omega} \ar[rru] \ar[rr] \ar[ldd] && \tilde{E}'_{\bar{W}^{2}, \cdots, V^{n}, \Omega} \ar[r] \ar[ldd] \ar[d] & E''_{V^{1}, \cdots, V^{n}, \Omega} \ar[d]
\\
&&& E'_{\bar{W}_{2}, \cdots, V^{n}, \Omega} \ar[r] \ar[ldd] & E''_{\bar{W}^{2}, \cdots, V^{n}, \Omega} \ar[dd] 
\\
E'_{V^{1}, V^{2}, \Omega} \times \prod_{k = 3}^{n}E_{V^{k}, \Omega} \ar[d] \ar[rr] && E''_{V^{1}, V^{2}, \Omega} \times \prod_{k = 3}^{n}E_{V^{k}, \Omega} \ar[d] &&
\\
E_{V^{1}, \Omega} \times E_{V^{2}, \Omega} \times \prod_{k = 3}^{n}E_{V^{k}, \Omega} &&  E_{\bar{W}^{2}, \Omega} \times \prod_{k = 3}^{n}E_{V^{k}, \Omega} && E_{V, \Omega}
}
\]
where 
\[
\tilde{E}'_{\bar{W}^{2}, \cdots, V^{n}, \Omega} = \Big\{(x , {\rm Fil}, \{\varphi_{k}\}_{k = 3}^{n}, \psi_{2}) | (x, {\rm Fil}) \in E''_{V^{1}, \cdots, V^{n}, \Omega}, \, \varphi_{k}: V^{k} \xrightarrow{\sim} W^{k}/W^{k-1} (k \geqslant 3), \, \psi_{2}: \bar{W}^{2} \xrightarrow{\sim} W^{2} \Big\}
\]
and
\[
\tilde{E}'_{V^{1}, \cdots, V^{n}, \Omega} = \Big\{ (x, {\rm Fil}, \{\varphi_{k}\}_{k = 1}^{n}, \psi_{2}) | (x, {\rm Fil}, \{\varphi_{k}\}_{k = 1}^{n}) \in E'_{V^{1}, \cdots, V^{n}, \Omega}, \, \psi_{2}: \bar{W}^{2} \xrightarrow{\sim} W^{2} \Big\}
\]
\[
\tilde{E}'_{V^{1}, \cdots, V^{n}, \Omega} \rightarrow E'_{V^{1}, V^{2}, \Omega} \times \prod_{k = 3}^{n}E_{V^{k}, \Omega}, \quad (x, {\rm Fil}, \{\varphi_{k}\}_{k = 1}^{n}, \psi_{2}) \mapsto \Big( (x', {\rm Fil}', \varphi'_{1}, \varphi'_{2}),  \{\varphi_{k}^{-1} x \varphi_{k}\}_{k = 3}^{n} \Big)
\]
where
\[
x' = \psi_{2}^{-1} x \psi_{2}, \quad {\rm Fil}': 0 = \psi_{2}^{-1}(W^{1}) \subsetneq \psi_{2}^{-1}(W^{2})  = \bar{W}^{2} 
\]
and 
\[\varphi'_{1} = \psi_{2}^{-1} \varphi_{1}, \quad \varphi'_{2} = \psi_{2}^{-1}\varphi_{2}
\]
In particular, 
\begin{align*}
\tilde{E}'_{\bar{W}^{2}, \cdots, V^{n}, \Omega} & \cong (E''_{V^{1}, V^{2}, \Omega} \times \prod_{k = 3}^{n}E_{V^{k}, \Omega}) \times_{(E_{\bar{W}^{2}, \Omega} \times \prod_{k = 3}^{n}E_{V^{k}, \Omega})} E'_{\bar{W}_{2}, \cdots, V^{n}, \Omega} \\
& \cong E'_{\bar{W}_{2}, \cdots, V^{n}, \Omega} \times_{E''_{\bar{W}^{2}, \cdots, V^{n}, \Omega}} E''_{V^{1}, \cdots, V^{n}, \Omega}
\end{align*}
and
\[
\tilde{E}'_{V^{1}, \cdots, V^{n}, \Omega} \cong (E'_{V^{1}, V^{2}, \Omega} \times \prod_{k = 3}^{n}E_{V^{k}, \Omega}) \times_{(E''_{V^{1}, V^{2}, \Omega} \times \prod_{k = 3}^{n}E_{V^{k}, \Omega})} \tilde{E}'_{\bar{W}^{2}, \cdots, V^{n}, \Omega}
\]
Let $\phi_{k} = \phi_{\mathcal{F}_{k}}$. It suffices to show 
\[
{\rm Ind}_{V^{1}, \cdots, V^{n}}(\mathcal{F}_{1} \boxtimes \cdots \boxtimes \mathcal{F}_{n}) = {\rm Ind}_{\bar{W}^{2}, \cdots, V^{n}}({\rm Ind}_{V^{1}, V^{2}}(\mathcal{F}_{1} \boxtimes \mathcal{F}_{2}) \boxtimes \cdots \boxtimes \mathcal{F}_{n}).
\]
This can be seen easily by tracing the diagram. 

\end{proof}

We will also define
\[
{\rm Ind}_{V^{1}, \cdots, V^{n}}: \otimes_{k = 1}^{n}M(\Lambda_{V^{k}})^{{G_{V^{k}}}} \rightarrow M(\Lambda_V)^{G_V}
\]
as follows. Consider the diagram
\begin{align}
\label{diag: induction conormal}
\xymatrix{
\Lambda'_{V^{1}, \cdots, V^{n}} \ar[r]^{q_{2}} \ar[d]_{q_{1}} & \Lambda''_{V^{1}, \cdots, V^{n}} \ar[d]^{q_{3}} \\
\Lambda_{V^{1}} \times \cdots \times \Lambda_{V^{n}} & \Lambda_{V}
}
\end{align}
where
\begin{align*}
\Lambda''_{V^{1}, \cdots V^{n}}  & := \Big\{(x, y, {\rm Fil}) | (x, y) \in \Lambda_{V}, \, {\rm Fil}: 0 = W^{0} \subsetneq \cdots \subsetneq W^{n} = V \text{ $(x, y)$-stable }  | W^{k}/W^{k-1}| = |V^{k}| \, (1 \leqslant k \leqslant n) \Big\} \\
\Lambda'_{V^{1}, \cdots V^{n}}   & := \Big\{(x, y, {\rm Fil}, \{\varphi_{k}\}_{k=1}^{n}) | (x, y, {\rm Fil}) \in \Lambda''_{V^{1}, \cdots V^{n}}  \text{ and } \varphi_{k}: V^{k} \xrightarrow{\sim} W^{k}/W^{k-1} \, (1 \leqslant k \leqslant n) \Big\} 
\end{align*}
and
\[
q_{3}: \Lambda''_{V^{1}, \cdots, V^{n}} \rightarrow \Lambda_{V}, \quad (x, y, {\rm Fil}) \mapsto (x, y)
\]
is proper;
\[
q_{2}: \Lambda'_{V^{1}, \cdots, V^{n}} \rightarrow \Lambda''_{V^{1}, \cdots, V^{n}}, \quad (x, y, {\rm Fil}, \{\varphi_{k}\}_{k=1}^{n}) \mapsto (x, y, {\rm Fil})
\]
is a principal $G_{V^{1}} \times \cdots \times G_{V^{n}}$-bundle;
\[
q_{1}: \Lambda'_{V^{1}, \cdots, V^{n}} \rightarrow \Lambda_{V^{1}} \times \cdots \times \Lambda_{V^{n}}, \quad (x, y, {\rm Fil}, \{\varphi_{k}\}_{k=1}^{n}) \mapsto \{(\varphi_{k}^{-1}x\varphi_{k}, \varphi_{k}^{-1}y\varphi_{k})\}_{k=1}^{n}
\]
where we denote the induced morphisms on $W^{k}/W^{k-1}$ still by $x, y$. 

For $\phi_{k} \in M(\Lambda_{V^{k}})^{G_{V^{k}}}$, we define 
\[
{\rm Ind}_{V^{1}, \cdots, V^{n}}(\phi_{1} \otimes \cdots \otimes \phi_{k}) (x, y) = \int_{q_{3}^{-1}(x, y)} \phi''(x, y, {\rm Fil})
\]
where 
\[
\phi''(x, y, {\rm Fil}) = \phi_{1}((x_{1}, y_{1})) \cdots \phi_{n}((x_{n}, y_{n}))
\]
with 
\[
x_{k} = \varphi_{k}^{-1} x \varphi_{k}, \, y_{k} = \varphi_{k}^{-1} y \varphi_{k}
\]
for any choice of isomorphisms $\varphi_{k}: V^{k} \rightarrow W^{k}/W^{k-1}$. 

One can easily extend this induction constructible functions on $E_{V^{k}}$ by considering the diagram
\[
\xymatrix{
& E'_{V^{1}, \cdots, V^{n}} \ar[rr]^{q'_{2}} \ar[dd]_{q'_{1}} && E''_{V^{1}, \cdots, V^{n}} \ar[dd]^{q'_{3}} \\
\Lambda'_{V^{1}, \cdots, V^{n}} \ar[ur] \ar[rr]^{\quad \quad \quad q_{2}} \ar[dd]_{q_{1}} && \Lambda''_{V^{1}, \cdots, V^{n}} \ar[ur] \ar[dd]^{q_{3}} & \\
& E_{V^{1}} \times \cdots \times E_{V^{n}} && E_{V} \\
\Lambda_{V^{1}} \times \cdots \times \Lambda_{V^{n}} \ar[ur] && \Lambda_{V} \ar[ur] & 
}
\]
where
\begin{align*}
E''_{V^{1}, \cdots, V^{n}}  & := \Big\{(x, y, {\rm Fil}) | (x, y) \in E_{V,
\Omega}, \, {\rm Fil}: 0 = W^{0} \subsetneq W^{1} \subsetneq W^{2} = V \text{ $(x, y)$-stable }  | W^{k}/W^{k-1}| = |V^{k}| \, (1 \leqslant k \leqslant n) \Big\} \\
E'_{V^{1}, \cdots, V^{n}}   & := \Big\{(x, y, {\rm Fil}, \{\varphi_{k}\}_{k=1}^{n}) | (x, y, {\rm Fil}) \in E''_{V^{1}, \cdots, V^{n}}  \text{ and } \varphi_{k}: V^{k} \xrightarrow{\sim} W^{k}/W^{k-1} \, (1 \leqslant k \leqslant n) \Big\} ,
\end{align*}
Note both the top and right squares are Cartesian, but the left one is not. The following lemma follows immediately from the definition. 

\begin{lemma}
For $\phi_{k} \in M(E_{V^{k},\Omega})^{G_{V^{k}}} \, (1 \leqslant k \leqslant n)$, 
\[
{\rm Ind}_{V^{1}, \cdots, V^{n}}(\phi_{1} \otimes \cdots \otimes \phi_{n})|_{\Lambda_{V}} = {\rm Ind}_{V^{1}, \cdots, V^{n}} (\phi_{1}|_{\Lambda_{V^{1}}} \otimes \cdots \otimes \phi_{n}|_{\Lambda_{V^{n}}}).
\]
\end{lemma}

For $\mathcal{F}_{k} \in D_{G_{V^{k}}}(E_{V^{k}}) \, (1 \leqslant k \leqslant n)$, we can also define induction 
\[
{\rm Ind}_{V^{1}, \cdots, V^{n}} (\mathcal{F}_{1} \boxtimes \cdots \boxtimes \mathcal{F}_{n}) := q'_{3 !} \mathcal{F}''
\]
where 
\[
q_{2}^{'*} \mathcal{F}'' \cong  q_{1}^{'*} (\mathcal{F}_{1} \boxtimes \cdots \boxtimes \mathcal{F}_{n}).
\]
One can also show easily that
\[
{\rm Ind}_{V^{1}, \cdots, V^{n}}(\phi_{\mathcal{F}_{1}} \otimes \cdots \otimes  \phi_{\mathcal{F}_{n}}) = \phi_{{\rm Ind}_{V^{1}, \cdots, V^{n}}(\mathcal{F}_{1} \boxtimes \cdots \boxtimes \mathcal{F}_{n})}
\]

\begin{prop}
\label{prop: convolution-induction conormal}
For $\phi_{k} \in M(\Lambda_{V^{k}})^{G_{V^{k}}} (k =1, \cdots, n)$, 
\[
\phi_{1} \ast \cdots \ast \phi_{n} = {\rm Ind}_{V^{1}, \cdots, V^{n}}(\phi_{1} \otimes \cdots \otimes \phi_{n}).
\]
\end{prop}

\begin{proof}

It suffices to show for $\phi_{k} \in M(E_{V^k})^{G_{V^{k}}} \, (1 \leqslant k \leqslant n)$,
\[
\phi_{1} \ast \cdots \ast \phi_{n} = {\rm Ind}_{V^{1}, \cdots, V^{n}}(\phi_{1} \otimes \cdots \otimes \phi_{n}).
\]
The proof is similar to that of Proposition~\ref{prop: convolution-induction}.

\end{proof}

\subsection{Further reduction of \eqref{eq: compatible with convolution}}\label{subsection-Further-reduction}

Let $|V| = (d_{i})_{i \in I}$ and $d = \sum_{i \in I} d_{i}$. Let 
\[
S_{|V|} := \{a \in I^{\{1, \cdots, d\}} \,| \, |a^{-1}(i)| = d_{i}\}
\]
For any $a \in S_{|V|}$, we fix an $I$-graded isomorphism $V = V^{1} \oplus \cdots \oplus V^{d}$ such that
\[
|V^{k}| = (\delta_{a(k), i})_{i \in I}
\]
Note $E_{V^{k}, \Omega} = \Lambda_{V^{k}} = E_{V^{k}} = \{0\}$. The following statement is a special case of \eqref{eq: compatible with convolution}.

\begin{conj}
\label{conj: main}
\begin{align}
\label{eq: main}
\eta_{V^{1}}(1_{1}) \ast \cdots \ast \eta_{V^{d}}(1_{d}) = \eta_{V}(1_{1} \ast \cdots \ast 1_{d}).
\end{align}
\end{conj}

\begin{rem}
$\eta_{V^{k}}(1_{k}) = 1_{\Lambda_{V^{k}}}$.
\end{rem}

Indeed, we have
\begin{lemma}
Conjecture~\ref{conj: main} is equivalent to \eqref{eq: compatible with convolution}.
\end{lemma}

\begin{proof}
We only need to show that \eqref{eq: compatible with convolution} follows from Conjecture~\ref{conj: main}. By \cite[Proposition 7.3]{Lu91}, it suffices to show \eqref{eq: compatible with convolution} for
\begin{align*}
\phi_{I} & = 1_{I, 1} \ast \cdots \ast 1_{I, d_{I}} \in M(E_{V^{I}, \Omega})^{G_{V^{I}}} \\
\phi_{II} & = 1_{II,1} \ast \cdots \ast 1_{II, d_{II}} \in M(E_{V^{II}, \Omega})^{G_{V^{II}}}
\end{align*}
associated with $a_{I} \in S_{|V^{I}|}$ and $a_{II} \in S_{|V^{II}|}$ respectively. By Conjecture~\ref{conj: main},
\begin{align*}
\eta_{V^{I}}(\phi_{I}) \ast \eta_{V^{II}}(\phi_{II}) & = (\eta_{V^{I, 1}}(1_{I, 1}) \ast \cdots \ast \eta_{V^{I, d_{I}}}(1_{I, 1})) \ast (\eta_{V^{II, 1}}(1_{II, 1}) \ast \cdots \ast \eta_{V^{II, d_{II}}}(1_{II, 1})) \\
& = \eta_{V}(1_{I, 1} \ast \cdots \ast 1_{I, d_{I}} \ast 1_{II,1} \ast \cdots \ast 1_{II, d_{II}}) \\
& = \eta_{V}(\phi_{I} \ast \phi_{II}) 
\end{align*}

\end{proof}

Now we will describe our approach to Conjecture~\ref{conj: main}. By Proposition~\ref{prop: convolution-induction} and Proposition~\ref{prop: convolution-induction conormal}, it suffices to show
\[
{\rm Ind}_{V^{1}, \cdots, V^{d}} (\eta_{V^{1}}(1_{1}) \otimes \cdots \otimes \eta_{V^{d}}(1_{d}) ) = \eta_{V}( {\rm Ind}_{V^{1}, \cdots, V^{d}} (1_{1} \ast \cdots \ast 1_{d}) ).
\]
Note $1_{k} = \phi_{\mathbbm{1}_{k}}$, where $\mathbbm{1}_{k} \in D_{G_{V^{k}}}(E_{V^{k}, \Omega})$. So it is the same as
\[
{\rm Ind}_{V^{1}, \cdots, V^{d}} (\eta_{V^{1}}(\phi_{\mathbbm{1}_{1}}) \otimes \cdots \otimes \eta_{V^{d}}(\phi_{\mathbbm{1}_{d}})) = \eta_{V}(f_{{\rm Ind}_{V^{1}, \cdots, V^{d}} (\mathbbm{1}_{1} \boxtimes \cdots \boxtimes \mathbbm{1}_{d})}).
\]
We first expand the right hand side. For $(x, y) \in \Lambda_{V}$, 
\[
\eta_{V}(f_{{\rm Ind}_{V^{1}, \cdots, V^{d}} (\mathbbm{1}_{1} \boxtimes \cdots \boxtimes \mathbbm{1}_{d})})(x, y) = \chi(R\Phi_{f_{y}}[-1](\mathbbm{1}_{1} \boxtimes \cdots \boxtimes \mathbbm{1}_{d})_{x}) = \chi(R\Phi_{f_{y}}[-1](p_{3!} \mathcal{F}'')_{x})
\]
By proper base change,
\[
R\Phi_{f_{y}}(p_{3!} \mathcal{F}'') \cong p_{3!}(R\Phi_{f_{y} \circ p_{3}}\mathcal{F}'')
\]
Hence,
\[
\chi(R\Phi_{f_{y}}(p_{3!} \mathcal{F}'')_{x}) = \chi(H^{*}(p_{3}^{-1}(x), R\Phi_{f_{y} \circ p_{3}}\mathcal{F}'') ) = \int_{p_{3}^{-1}(x)} \chi(R\Phi_{f_{y} \circ p_{3}}\mathcal{F}'').
\]
By smooth base change,
\[
p_{2}^{*}(R\Phi_{f_{y} \circ p_{3}}\mathcal{F}'') \cong R\Phi_{f_{y} \circ p_{3} \circ p_{2}} (p_{2}^{*}\mathcal{F}'') \cong R\Phi_{f_{y} \circ p_{3} \circ p_{2}} (\mathbbm{1}).
\]
We can also express the left hand side as an integration,
\begin{align*}
{\rm Ind}_{V^{1}, \cdots, V^{d}} (\eta_{V^{1}}(\phi_{\mathbbm{1}_{1}}) \otimes \cdots \otimes \eta_{V^{d}}(\phi_{\mathbbm{1}_{d}}))  (x, y) = &  \int_{q_{3}^{-1}(x, y)} \eta_{V^{1}}(\phi_{\mathbbm{1}_{1}})(x_{1}, y_{1}) \, \cdots \, \eta_{V^{d}}(\phi_{\mathbbm{1}_{d}})(x_{d}, y_{d})\\
 = & \int_{q_{3}^{-1}(x, y)} 1,
\end{align*}
due to the fact that $(x_{k}, y_{k}) = 0$ for $1 \leqslant k \leqslant d$. Comparing the two integrals, we see
\[
q_{3}^{-1}(x, y) \hookrightarrow p_{3}^{-1}(x), \quad (x, y, {\rm Fil}) \mapsto (x, {\rm Fil}).
\]
If we want to prove that the two integrals are equal, it suffices to show 
\begin{align}
\label{eq: vanishing}
R\Phi_{f_{y} \circ p_{3} \circ p_{2}} (\mathbbm{1})_{(x, {\rm Fil}, \{\varphi_{k}\}_{k=1}^{d})} = 0 \quad \quad  \text{ for } (x, {\rm Fil}) \in p_{3}^{-1}(x) \backslash q_{3}^{-1}(x, y),
\end{align}
\begin{align}
\label{eq: euler}
\chi(R\Phi_{f_{y} \circ p_{3} \circ p_{2}}[-1] (\mathbbm{1}))_{(x, {\rm Fil}, \{\varphi_{k}\}_{k=1}^{d})} = 1 \quad \quad \text{ for } (x, {\rm Fil}) \in q_{3}^{-1}(x, y).
\end{align}

To show these, we adopt the diagram 
\begin{align*}
\xymatrix{
G_{V} \times E_{V^{1}, \cdots, V^{d}, \Omega}^{\geqslant 0} \ar[d]_{p'_{0}} & \\
G_{V} \times_{G_{V^{1}, \cdots, V^{d}}^{+}} E_{V^{1}, \cdots, V^{d}, \Omega}^{\geqslant 0} \ar[r]^{p'_{2}} \ar[d]_{p'_{1}} & G_{V} \times_{G_{V^{1}, \cdots, V^{d}}^{\geqslant 0}} E_{V^{1}, \cdots, V^{d}, \Omega}^{\geqslant 0} \ar[d]^{p'_{3}} \\
E_{V^{1}, \Omega} \times \cdots \times E_{V^{d}, \Omega} & E_{V, \Omega}
}
\end{align*}
Let 
\[
G_{V} \times E_{V^{1}, \cdots, V^{d}, \Omega}^{\geqslant 0} \ni (g_{0}, x_{0}) \mapsto (x, {\rm Fil}, \{\varphi_{k}\}_{k=1}^{d})
\] 
and 
\[
h_{y}: G_{V} \times E_{V^{1}, \cdots, V^{d}, \Omega}^{\geqslant 0} \rightarrow \mathbb{C}, \quad (g, x') \mapsto \langle gx', y\rangle
\]
be the pullback of $f_{y} \circ p_{3} \circ p_{2}$ along $p'_{0}$. Then by smooth base change,
\[
R\Phi_{f_{y} \circ p_{3} \circ p_{2}} (\mathbbm{1})_{(x, {\rm Fil}, \{\varphi_{k}\}_{k=1}^{d})} \cong R\Phi_{h_{y}}(\mathbbm{1})_{(g_{0}, x_{0})}.
\]
If $R\Phi_{h_{y}}(\mathbbm{1})_{(g_{0}, x_{0})} \neq 0$, then $h$ is singular at $(g_{0}, x_{0})$. So we compute
\[
dh|_{(g_{0}, x_{0})}: T_{(g_{0}, x_{0})} \, (G_{V} \times E_{V^{1}, \cdots, V^{d}, \Omega}^{\geqslant 0}) \cong \mathfrak{g}_{V} \times E_{V^{1}, \cdots, V^{d}, \Omega}^{\geqslant 0} \rightarrow \mathbb{C}, \quad (u, v) \mapsto \langle [u, x], y \rangle + \langle g_{0}v, y \rangle.
\]
Since $(x, y) \in \Lambda_{V}$, $\langle [u, x], y \rangle = 0$. So $dh|_{(g_{0}, x_{0})} = 0$ if and only if $\langle g_{0}v, y\rangle = 0$ for all $v \in E_{V^{1}, \cdots, V^{d}, \Omega}^{\geqslant 0}$. Since
\[
\langle g_{0}v, y \rangle = \langle v, g_{0}^{-1}y \rangle,
\]
this is also equivalent to require that $g_{0}^{-1}y$ stabilizes $\overline{{\rm Fil}}$, which is the same to say $y$ stabilizes $g_{0}\overline{{\rm Fil}} = {\rm Fil}$, i.e., $(x, {\rm Fil}) \in q_{3}^{-1}(x, y)$. So we have shown \eqref{eq: vanishing}.

We are now left with \eqref{eq: euler}. Assume $(x, {\rm Fil}) \in q_{3}^{-1}(x, y)$, i.e., $y_{0} := g_{0}^{-1}y \in E_{V^{1}, \cdots, V^{d}, \bar{\Omega}}^{\geqslant 0}$. By applying $g_{0}$, we get
\[
\xymatrix{
(g, x') \ar[dd] & G_{V} \times E_{V^{1}, \cdots, V^{d}, \Omega}^{\geqslant 0} \ar[dd] \ar[rd]_{h_{y_{0}}} & \\
& & \mathbb{C} \\
(g_{0}g, x')& G_{V} \times E_{V^{1}, \cdots, V^{d}, \Omega}^{\geqslant 0} \ar[ur]^{h_{y}}
}
\]
where $R\Phi_{h_{y}}(\mathbbm{1})_{(g_{0}, x_{0})} \cong R\Phi_{h_{y_{0}}}(\mathbbm{1})_{(1, x_{0})}$. So we have reduced it to the following statement.

\begin{conj}
\label{conj: vanishing cycle}
$\chi(R\Phi_{h_{y_{0}}}[-1](\mathbbm{1}))_{(1, x_{0})} = 1$.
\end{conj}

In the next section, we will prove this for type $A_2$ quiver.

\section{Quiver of type $A_2$}
\label{sec: A_2}

Let $V = V_{1} \oplus V_{2}$ be a graded vector space and $\Omega$ be the orientation $1 \rightarrow 2$. Let $d_{i} = {\rm dim}V_{i}$ and $d_{1} + d_{2} = d$.
\begin{align*}
E_{V, \Omega} = {\rm Hom}(V_{1}, V_{2}), \quad E_{V, \bar{\Omega}} = {\rm Hom}(V_{2}, V_{1}), \quad G_{V} = GL(V_{1}) \times GL(V_{2}).
\end{align*}
For $(x, y) \in E_{V, \Omega} \times E_{V, \bar{\Omega}}$ and $g = (g_{1}, g_{2}) \in G_{V}$, we have the group action
\[
g \cdot x = g_{2} x g_{1}^{-1}, \quad \quad g \cdot y = g_{1} y g_{2}^{-1}
\]
We also have the Lie bracket
\[
[x, y] = (-yx, xy) \in {\rm End}(V_{1}) \times {\rm End}(V_{2})
\]
and $G_{V}$-invariant nondegenerate pairing
\[
\langle x, y \rangle = {\rm tr}(xy).
\]

We fix an $I$-graded isomorphism $V = V^{1} \oplus V^{2} \oplus \cdots \oplus V^{d}$ such that ${\rm dim} V^{k} = 1$. Then
\[
V_{1} = V^{t_{1}} \oplus \cdots \oplus V^{t_{d_{1}}}, \quad \quad V_{2} = V^{s_{1}} \oplus \cdots \oplus V^{s_{d_{2}}},
\]
where $t_{1} < \cdots < t_{d_{1}}$ and $s_{1} < \cdots < s_{d_{2}}$. Select all indexes 
\begin{align*}
1 = \mu_{1} < \mu_{2} < \cdots < \mu_{e} = d_{1} + 1, \\
1 = \nu_{1} < \nu_{2} < \cdots < \nu_{f} = d_{2} + 1.
\end{align*}
such that 
\begin{align*}
t_{\mu_{i} - 1} + 1 < t_{\mu_{i}} = t_{\mu_{i} + 1} - 1, \quad \text{ for } i < e; \\
s_{\nu_{i} - 1} + 1 < s_{\nu_{i}} = s_{\nu_{i} + 1} - 1,\quad \text{ for } i < f.
\end{align*}
Then the list $\{1, 2, \cdots, d\}$ would correspond to either of these cases below:
\begin{enumerate}
\item
\[
s_{\nu_{1}}, \cdots, s_{\nu_{2} - 1}; t_{\mu_{1}}, \cdots t_{\mu_{2} - 1}; s_{\nu_{2}} \cdots \cdots; s_{\nu_{f-1}}, \cdots, s_{\nu_{f} - 1}; t_{\mu_{e-1}} \cdots t_{\mu_{e} - 1} ,\quad \quad e = f,
\]
\item
\[
s_{\nu_{1}}, \cdots, s_{\nu_{2} - 1}; t_{\mu_{1}}, \cdots t_{\mu_{2} - 1}; s_{\nu_{2}} \cdots \cdots; t_{\mu_{e-1}} \cdots t_{\mu_{e} - 1}; s_{\nu_{f-1}}, \cdots, s_{\nu_{f} - 1}  ,\quad \quad e = f - 1,
\]
\item
\[
t_{\mu_{1}}, \cdots t_{\mu_{2} - 1}; s_{\nu_{1}}, \cdots, s_{\nu_{2} - 1}; t_{\mu_{2}} \cdots \cdots; t_{\mu_{e-1}} \cdots t_{\mu_{e} - 1}; s_{\nu_{f-1}}, \cdots, s_{\nu_{f} - 1}, \quad \quad e = f,
\]
\item
\[
t_{\mu_{1}}, \cdots t_{\mu_{2} - 1}; s_{\nu_{1}}, \cdots, s_{\nu_{2} - 1}; t_{\mu_{2}} \cdots \cdots; s_{\nu_{f-1}}, \cdots, s_{\nu_{f} - 1}; t_{\mu_{e-1}} \cdots t_{\mu_{e} - 1} ,\quad \quad e = f + 1.
\]
\end{enumerate}
We fix a filtration
\[
\overline{{\rm Fil}}: 0 = \bar{W}^{0} \subsetneq \bar{W}^{1} \subsetneq \cdots \subsetneq \bar{W}^{d} = V
\]
where 
\[
\bar{W}^{k} := V^{1} \oplus \cdots \oplus V^{k}.
\]
Let
\[
E_{V^{1}, \cdots, V^{d}, \Omega}^{\geqslant 0} := \{x \in E_{V, \Omega} | x \text{ stabilizes } \overline{{\rm Fil}} \}.
\]
It admits an action by 
\[
G_{V^{1}, \cdots, V^{d}}^{\geqslant 0} := \{g \in G_{V} | g \text{ stabilizes } \overline{{\rm Fil}} \}
\]
a Borel subgroup of $G_{V}$. It has a Levi component $G_{V^{1}} \times \cdots \times G_{V^{d}}$, which is a maximal torus, and the unipotent radical is
\[
G_{V^{1}, \cdots, V^{d}}^{+} := \{g \in G_{V^{1}, \cdots, V^{d}}^{\geqslant 0} | \bar{\varphi}_{k}^{-1} g \bar{\varphi}_{k} = id \text{ for } k = 1, \cdots, d\},
\]
where $\bar{\varphi}_{k}: V^{k} \hookrightarrow W^{k} \mapsto W^{k}/W^{k-1}$. Similarly we can define $E_{V^{1}, \cdots, V^{d}, \bar{\Omega}}^{\geqslant 0}$.

We also fix basis vectors $v_{k}$ for $V^{k}$, then they give a basis for each $V_{i}$. Under these basis, we have
\begin{align*}
E_{V, \Omega} & = {\rm Hom}(V_{1}, V_{2}) \cong {\rm Mat}_{d_{2} \times d_{1}}(\mathbb{C}),\\
E_{V, \bar{\Omega}} & = {\rm Hom}(V_{2}, V_{1}) \cong {\rm Mat}_{d_{1} \times d_{2}}(\mathbb{C}),\\
G_{V} & = GL(V_{1}) \times GL(V_{2}) \cong GL(d_{1}, \mathbb{C}) \times GL(d_{2}, \mathbb{C}).
\end{align*}
Let $B_{i}$ be the Borel subgroup of $GL(d_{i}, \mathbb{C})$, consisting of upper triangular matrices with unipotent radical $U_{i}$, then
\[
G_{V^{1}, \cdots, V^{d}}^{\geqslant 0} \cong B_{1} \times B_{2}, \quad \quad G_{V^{1}, \cdots, V^{d}}^{+} \cong U_{1} \times U_{2}.
\]

To describe the elements of $E_{V^{1}, \cdots, V^{d}, \Omega}^{\geqslant 0}$ in terms of matrices, we should turn a $d_{2} \times d_{1}$ matrix into a block matrix by requring the first row in $k$-th row block is row $\nu_{k}$ and the the first column in $k$-th column block is column $\mu_{k}$. There are $f-1$ row blocks and $e-1$ column blocks. Depending on the previous four cases, we will get the following shape of elements $x = \{X_{i, j}\} \in E_{V^{1}, \cdots, V^{d}, \Omega}^{\geqslant 0}$. We will use $\ast$ and $0$ to indicate the blocks. In case (1),
\[
\begin{pmatrix}
* & * & \cdots & * & * \\
0 & * & \cdots & * & * \\
\vdots & \vdots & \ddots & \vdots & * \\
0 & 0 & \cdots & * & * \\
0 & 0 & \cdots & 0 & * \\
\end{pmatrix}.
\]
In case (2),
\[
\begin{pmatrix}
* & * & \cdots & *  \\
0 & * & \cdots & * \\
\vdots & \vdots & \ddots & \vdots \\
0 & 0 & \cdots & * \\
0 & 0 & \cdots & 0 \\
\end{pmatrix}.
\]
In these two cases, if $X_{i,j} \neq 0$, then
\[
\nu_{k} \leqslant i < \nu_{k+1} \text{ for $k < f$ } \Rightarrow j \geqslant \mu_{k}.
\]
In case (3),
\[
\begin{pmatrix}
0 & * & \cdots & * \\
\vdots & \vdots & \ddots & \vdots \\
0 & 0 & \cdots & * \\
0 & 0 & \cdots & 0 \\
\end{pmatrix}.
\]
In case (4),
\[
\begin{pmatrix}
0 & * & \cdots & * & * \\
\vdots & \vdots & \ddots & \vdots & * \\
0 & 0 & \cdots & * & * \\
0 & 0 & \cdots & 0 & * \\
\end{pmatrix}.
\]
In these two cases, if $X_{i,j} \neq 0$, then
\[
\nu_{k} \leqslant i < \nu_{k+1} \text{ for $k < f$ } \Rightarrow j \geqslant \mu_{k + 1}.
\]
Similarly, for $y = \{Y_{i, j}\} \in E_{V^{1}, \cdots, V^{d}, \bar{\Omega}}^{\geqslant 0}$. In case (1) and (2), if $Y_{i,j} \neq 0$, then
\[
\mu_{k} \leqslant i < \mu_{k+1} \text{ for $k < e$ } \Rightarrow j \geqslant \nu_{k + 1}.
\]
In case (3) and (4), if $Y_{i,j} \neq 0$, then
\[
\mu_{k} \leqslant i < \mu_{k+1} \text{ for $k < e$ } \Rightarrow j \geqslant \nu_{k}.
\]

\subsection{Vanishing cycle}
\label{sec: vanishing cycle}

For $y_{0} \in E_{V^{1}, \cdots, V^{d}, \bar{\Omega}}^{\geqslant 0}$, we want to compute the vanishing cycle of
\[
h_{y_{0}}: G_{V} \times E_{V^{1}, \cdots, V^{d}, \Omega}^{\geqslant 0} \rightarrow \mathbb{C}, (g, x) \mapsto \langle gx, y_{0} \rangle
\]
at $(1, x_{0})$, where $(x_{0}, y_{0}) \in \Lambda_{V}$. First we would like to show that it suffices to consider those $y_{0}$ in nice shape. Suppose $y'_{0} = g' \cdot y_{0}$ for $g' \in G_{V^{1}, \cdots, V^{d}}^{\geqslant 0}$, then
\[
\xymatrix{
(g, x) \ar[dd] & G_{V} \times E_{V^{1}, \cdots, V^{d}, \Omega}^{\geqslant 0} \ar[dd] \ar[rd]_{h_{y_{0}}} & \\
& & \mathbb{C} \\
(g'gg'^{-1}, g' \cdot x) & G_{V} \times E_{V^{1}, \cdots, V^{d}, \Omega}^{\geqslant 0} \ar[ur]^{h_{y'_{0}}}
}
\]
Let $x'_{0} = g' \cdot x_{0}$, then it is the same to consider the vanishing cycle of $h_{y'_{0}}$ at $(1, x'_{0})$. So we can change $y_0$ by the action of $G_{V^{1}, \cdots, V^{d}}^{\geqslant 0} \cong B_1 \times B_2$. Note the action of $B_1$ on $y_0$ is by row operations and the action of $B_2$ on $y_0$ is by column operations. For each nonzero column $s$ of $y_{0}$, let $\alpha_{s}$ be the first nonzero entry from the bottom. By the action of $B_{1}$, we can make all entries above $\alpha_{s}$ be zero. Then by the action of $B_{2}$, we can make all entries on the right of $\alpha_{s}$ be zero. If we do this process from the first column to the last column, then we can make {\bf each row and column of $y_0$ contain at most one nonzero entry, which can be further normalized to be one}. From now on, we will assume $y_{0}$ satisfies this property. Let us index the nonzero entries in $y_{0}$ by a set $A$ and $\alpha \in A$ corresponds to the entry $(i_{\alpha}, j_{\alpha})$. Let
\(
I = \{i_{\alpha} | \alpha \in A \}, J = \{j_{\alpha} | \alpha \in A\}.
\)
Our choice of $y_0$ has the following consequence on $x_0$.

\begin{lemma}
\label{lemma: projection to zero}
$X_{ij} = 0$ for $j \in I$ or $i \in J$ at $x_{0}$
\end{lemma}

\begin{proof}
Since $[x_{0}, y_{0}] = 0$, then $x_{0}y_{0} = y_{0}x_{0} = 0$. The result follows immediately from our assumption on $y_{0}$.
\end{proof}

Let $\bar{U}_{i}$ the unipotent radical of the opposite Borel subgroup $\bar{B}_{i}$. Then the map
\[
(\bar{U}_{1} \times \bar{U}_{2}) \times (B_{1} \times B_{2}) \rightarrow GL(d_{1}, \mathbb{C}) \times GL(d_{2}, \mathbb{C}), \quad (u, b) \mapsto ub
\]
is smooth. By smooth base change, it suffices to consider the vanishing cycle of the pullback
\[
\tilde{h}_{y_{0}}: (\bar{U}_{1} \times \bar{U}_{2}) \times (B_{1} \times B_{2}) \times E_{V^{1}, \cdots, V^{d}, \Omega}^{\geqslant 0} \rightarrow \mathbb{C}
\]
One can also consider the composition of the projection
\[
{\rm pr}: (\bar{U}_{1} \times \bar{U}_{2}) \times (B_{1} \times B_{2}) \times E_{V^{1}, \cdots, V^{d}, \Omega}^{\geqslant 0} \rightarrow (\bar{U}_{1} \times \bar{U}_{2}) \times E_{V^{1}, \cdots, V^{d}, \Omega}^{\geqslant 0}
\]
with the restriction of $h_{y_{0}}$ to $(\bar{U}_{1} \times \bar{U}_{2}) \times E_{V^{1}, \cdots, V^{d}, \Omega}^{\geqslant 0}$, and we denote it by $\bar{h}_{y_{0}}$. Then we have a commutative diagram
\[
\xymatrix{
(u, b, x) \ar[dd] & (\bar{U}_{1} \times \bar{U}_{2}) \times (B_{1} \times B_{2}) \times E_{V^{1}, \cdots, V^{d}, \Omega}^{\geqslant 0} \ar[dd] \ar[rd]_{\tilde{h}_{y_{0}}} & \\
& & \mathbb{C} \\
(u, b, b \cdot x) & (\bar{U}_{1} \times \bar{U}_{2}) \times (B_{1} \times B_{2}) \times E_{V^{1}, \cdots, V^{d}, \Omega}^{\geqslant 0} \ar[ur]^{\bar{h}_{y_{0}}}
}
\]
So it is the same to consider $\bar{h}_{y_{0}}$. Finally, since $\bar{h}_{y_{0}}$ factors through ${\rm pr}$, by smooth base change it suffices to consider the restriction of $h_{y_{0}}$ to $(\bar{U}_{1} \times \bar{U}_{2}) \times E_{V^{1}, \cdots, V^{d}, \Omega}^{\geqslant 0}$, and we denote it still by $\bar{h}_{y_{0}}$. So we have shown

\begin{lemma}
$R\Phi_{h_{y_{0}}}(\mathbbm{1})_{(1, x_{0})} \cong R\Phi_{\bar{h}_{y_{0}}}(\mathbbm{1})_{(1, x_{0})}$.
\end{lemma}

Let us denote the entries of $g_{1}^{-1}$ by $M_{ij}$ and that of $g_{2}$ by $N_{ij}$. We want to calculate $\bar{h}_{y_{0}}$ explicitly,
\begin{align*}
\bar{h}_{y_{0}}((g_{1}, g_{2}), x) & = {\rm tr}(xg_{1}^{-1}y_{0}g_{2}) \\
& = \sum_{i, j} \sum_{s > i} X_{ij} Y_{js} N_{si} + \sum_{i, j} \sum_{r < j} X_{ij}M_{jr}Y_{rs} + \sum_{i, j} \sum_{\substack{r < j \\ s > i}} X_{ij}M_{jr}Y_{rs}N_{si} \\
& = \sum_{\alpha \in A, i}X_{ii_{\alpha}} N_{j_{\alpha}i} + \sum_{\alpha' \in A, j} X_{j_{\alpha'}j}M_{ji_{\alpha'}} + \sum_{\alpha'' \in A, i, j} X_{ij}M_{ji_{\alpha''}}N_{j_{\alpha''}i}.
\end{align*}
There are three terms in the summation. We first consider $X_{ij}$ appearing in both of the first two terms. They are necessarily of the form $X_{j_{\alpha'}i_{\alpha}}$. Let us define
\[
T = \{(\alpha', \alpha) \in A^{2}| X_{j_{\alpha'}i_{\alpha}} \neq 0\}.
\]
We have an inclusion
\[
\pi: T \rightarrow J \times I, \quad (\alpha', \alpha) \mapsto (j_{\alpha'}, i_{\alpha}).
\]

\begin{lemma}
For any $(\alpha', \alpha) \in T$, $X_{j_{\alpha'}i_{\alpha}}$ appears in both of the first two terms.
\end{lemma}

\begin{proof}
By the shape of $x$ and $y_{0}$, we know if $X_{j_{\alpha'}i_{\alpha}} \neq 0$, then $j_{\alpha} > j_{\alpha'}$ and $i_{\alpha} > i_{\alpha'}$. The rest is clear.
\end{proof}

Combining the terms with $X_{j_{\alpha'}, i_{\alpha}}$ for $(\alpha', \alpha) \in T$, we get
\begin{align}
\label{eq: X}
\sum_{(\alpha', \alpha) \in T} X_{j_{\alpha'}, i_{\alpha}}\Big(N_{j_{\alpha}j_{\alpha'}} + M_{i_{\alpha}i_{\alpha'}} + \sum_{\alpha'' \in A : \, \substack{i_{\alpha''} < i_{\alpha} \\ j_{\alpha''} > j_{\alpha'}}} M_{i_{\alpha}i_{\alpha''}}N_{j_{\alpha''}j_{\alpha'}}\Big).
\end{align}
The remaining terms are
\begin{align}
\label{eq: N}
\sum_{\alpha \in A, i \notin J} X_{ii_{\alpha}}N_{j_{\alpha}i}
\end{align}

\begin{align}
\label{eq: M}
\sum_{\alpha' \in A, j \notin I} X_{j_{\alpha'}j}M_{ji_{\alpha'}}
\end{align}

\begin{align}
\label{eq: MN}
\sum_{\alpha'' \in A, i \notin J, j \notin I} X_{ij} M_{ji_{\alpha''}}N_{j_{\alpha''}i}
\end{align}

\begin{align}
\label{eq: XM}
\sum_{\alpha'' \in A, i \in J, j \notin I} X_{ij} M_{ji_{\alpha''}}N_{j_{\alpha''}i}
\end{align}

\begin{align}
\label{eq: XN}
\sum_{\alpha'' \in A, i \notin J, j \in I} X_{ij} M_{ji_{\alpha''}}N_{j_{\alpha''}i}.
\end{align}

The goal is to separate the variables so that the function can be viewed as a quadratic form in some variables (called {\bf quadratic variables}) with coefficients in different variables (called {\bf coefficient variables}). First we set $M_{ji}$ $(j \notin I, i \in I)$, $N_{ji}$ $(j \in J, i \notin J)$ and $X_{ij}$ $(i \in J, j \notin I)$ or $(i \notin J, j \in I)$ quadratic variables; set $N_{ji}$ $(j \in J, i \in J)$, $X_{ij}$ $(i \notin J, j \notin I)$ coefficient variables. Note the variables $M_{ji}$ $(i \notin I)$, $N_{ji}$ $(j \notin J)$ will never appear, so we can set them arbitrary. (We will set them as coefficient variables if not specified.) It remains to deal with 
\[
M_{ji} \quad (j \in I, i \in I) 
\]
which only appear in \eqref{eq: X} and \eqref{eq: XN}, and 
\[
X_{ij} \quad (i \in J, j \in I)
\]
which only appear in \eqref{eq: X}. Note in the latter case, $X_{ij} \neq 0$ only when $(i, j) \in \pi(T)$. To achieve our goal, we need to make some change of variables. For any $(\alpha', \alpha) \in T$,
\begin{align*}
M'_{i_{\alpha}i_{\alpha'}} & = N_{j_{\alpha}j_{\alpha'}} + M_{i_{\alpha}i_{\alpha'}} + \sum_{\alpha'' \in A: \, \substack{i_{\alpha''} < i_{\alpha} \\ j_{\alpha''} > j_{\alpha'}}} M_{i_{\alpha}i_{\alpha''}}N_{j_{\alpha''}j_{\alpha'}} \\
& = M_{i_{\alpha}i_{\alpha'}} + \Big(N_{j_{\alpha}j_{\alpha'}} + \sum_{\alpha'' \in A :\, \substack{i_{\alpha''} < i_{\alpha} \\ j_{\alpha''} > j_{\alpha'}}} M_{i_{\alpha}i_{\alpha''}}N_{j_{\alpha''}j_{\alpha'}}\Big).
\end{align*}
To see this is well-defined, we impose a partial order on $T$ such that
\[
(\alpha', \alpha) >_{T} (\alpha'', \alpha) \quad \text{ if } j_{\alpha'} < j_{\alpha''}
\]
Then
\begin{align}
\label{eq: inversion}
M_{i_{\alpha}i_{\alpha'}} = M'_{i_{\alpha}i_{\alpha'}} + \sum_{\substack{ (\alpha'', \alpha) \in T\\ (\alpha'', \alpha) <_{T} (\alpha', \alpha)}} \mathcal{U}^{i_{\alpha}}_{j_{\alpha'}, j_{\alpha''}} M'_{i_{\alpha}i_{\alpha''}} + \mathcal{U}^{i_{\alpha}}_{j_{\alpha'}}
\end{align}
where $\mathcal{U}^{i_{\alpha}}_{j_{\alpha'}, j_{\alpha''}}$ are polynomials in $N_{ij}$ $(i, j \in J)$, and $\mathcal{U}^{i_{\alpha}}_{j_{\alpha'}}$ are polynomials in $N_{ij}$ $(i, j \in J)$ and $M_{i_{\alpha}i_{\alpha''}}$ for any $(\alpha'', \alpha) \notin T$). Set $\mathcal{U}^{i_{\alpha}}_{j_{\alpha'}, j_{\alpha'}} = 1$. After this change of variables, \eqref{eq: X} becomes
\begin{align}
\label{eq: X'}
\sum_{(\alpha', \alpha) \in T} X_{j_{\alpha'}i_{\alpha}}M'_{i_{\alpha}i_{\alpha'}}.
\end{align}
We can also split \eqref{eq: XN} into two parts:
\[
\eqref{eq: XN}a:  \sum_{(\alpha'', \alpha') \notin T, i \notin J} X_{ii_{\alpha'}} M_{i_{\alpha'}i_{\alpha''}}N_{j_{\alpha''}i}
\]
and
\[
\eqref{eq: XN}b : \sum_{(\alpha'', \alpha') \in T, i \notin J} X_{ii_{\alpha'}} M_{i_{\alpha'}i_{\alpha''}}N_{j_{\alpha''}i}
\]
Substitute \eqref{eq: inversion} into \eqref{eq: XN}b, we get
\begin{align*}
& \sum_{(\alpha', \alpha) \in T, i \notin J} X_{ii_{\alpha}} M_{i_{\alpha}i_{\alpha'}}N_{j_{\alpha'}i} \\
& =  \sum_{(\alpha' , \alpha) \in T, i \notin J}  X_{ii_{\alpha}} \Big( M'_{i_{\alpha}i_{\alpha'}} + \sum_{\substack{ (\alpha'', \alpha) \in T\\ (\alpha'', \alpha) <_{T} (\alpha', \alpha)}} \mathcal{U}^{i_{\alpha}}_{j_{\alpha'}, j_{\alpha''}} M'_{i_{\alpha}i_{\alpha''}} + \mathcal{U}^{i_{\alpha}}_{j_{\alpha'}} \Big) N_{j_{\alpha'}i} \\
& = \sum_{(\alpha'', \alpha) \in T} \Big( \sum_{\substack{ (\alpha', \alpha) \in T \\ (\alpha', \alpha) \geqslant_{T} (\alpha'', \alpha), i \notin J}} X_{ii_{\alpha}}\mathcal{U}^{i_{\alpha}}_{j_{\alpha'}, j_{\alpha''}}N_{j_{\alpha'}i} \Big)M'_{i_{\alpha}i_{\alpha''}} + \sum_{(\alpha' ,\alpha) \in T, i \notin J} X_{ii_{\alpha}} \mathcal{U}^{i_{\alpha}}_{j_{\alpha'}}  N_{j_{\alpha'}i} 
\end{align*}
Combined with \eqref{eq: X'}, we get
\[
\sum_{(\alpha' ,\alpha) \in T} \Big(X_{j_{\alpha'}i_{\alpha}} + \sum_{\substack{ (\alpha'', \alpha) \in T \\ (\alpha'', \alpha) \geqslant_{T} (\alpha', \alpha), i \notin J}} X_{ii_{\alpha}}\mathcal{U}^{i_{\alpha}}_{j_{\alpha''}, j_{\alpha'}}N_{j_{\alpha''}i} \Big)M'_{i_{\alpha}i_{\alpha'}} + \sum_{(\alpha' ,\alpha) \in T, i \notin J}  X_{ii_{\alpha}} \mathcal{U}^{i_{\alpha}}_{j_{\alpha'}}  N_{j_{\alpha'}i}.
\]
For $(\alpha', \alpha) \in T$, let
\begin{align}
\label{eq: change X}
X'_{j_{\alpha'}i_{\alpha}} = X_{j_{\alpha'}i_{\alpha}} + \sum_{\substack{(\alpha'', \alpha) \in T \\ (\alpha'', \alpha) \geqslant_{T} (\alpha', \alpha), i \notin J}} X_{ii_{\alpha}}\mathcal{U}^{i_{\alpha}}_{j_{\alpha''}, j_{\alpha'}} N_{j_{\alpha''}i}.
\end{align}
Substitute $X'_{j_{\alpha'}i_{\alpha}}$ into the previous expression, we get
\begin{align}
\label{eq: X'+XNb}
\sum_{(\alpha' ,\alpha) \in T} X'_{j_{\alpha'}i_{\alpha}} M'_{i_{\alpha}i_{\alpha'}} + \sum_{(\alpha' ,\alpha) \in T, i \notin J} X_{ii_{\alpha}} \mathcal{U}^{i_{\alpha}}_{j_{\alpha'}}  N_{j_{\alpha'}i}.
\end{align}
In sum, after the substitutions by $M'_{i_{\alpha}i_{\alpha'}}$ and $X'_{j_{\alpha'}i_{\alpha}}$ for all $(\alpha', \alpha) \in T$, we see
\[
\eqref{eq: X} + \eqref{eq: XN}b = \eqref{eq: X'+XNb},
\]
Combined with \eqref{eq: XN}a, we rewrite them as
\begin{align*}
\eqref{eq: X}': \sum_{(\alpha' ,\alpha) \in T} X'_{j_{\alpha'}i_{\alpha}} M'_{i_{\alpha}i_{\alpha'}}
\end{align*}
and
\begin{align*}
\eqref{eq: XN}':  \sum_{\substack{i \notin J, (\alpha'', \alpha') \notin T}} X_{ii_{\alpha'}} M_{i_{\alpha'}i_{\alpha''}}N_{j_{\alpha''}i} + \sum_{\substack{i \notin J, (\alpha'', \alpha') \in T}} X_{ii_{\alpha'}} \mathcal{U}^{i_{\alpha'}}_{j_{\alpha''}}  N_{j_{\alpha''}i}
\end{align*}
so
\[
{\rm tr}(xg_{1}^{-1}y_{0}g_{2}) = \eqref{eq: X}' + \eqref{eq: N} + \eqref{eq: M} + \eqref{eq: MN} + \eqref{eq: XM} + \eqref{eq: XN}'.
\]
We will set $M'_{i_{\alpha}i_{\alpha'}}$ and $X'_{j_{\alpha'}i_{\alpha}}$ for $(\alpha', \alpha) \in T$ as quadratic variables. We will also set the variables $M_{i_{\alpha}i_{\alpha'}}$ for $(\alpha', \alpha) \notin T$ as coefficient variables.

\subsection{Euler characteristic of Milnor fiber}\label{subsection-Euler}

We want to compute 
\[
\chi(R\Phi_{\bar{h}_{y_{0}}}[-1](\mathbbm{1})_{(1, x_{0})}) = 1- \chi(R\Psi_{\bar{h}_{y_{0}}}(\mathbbm{1})_{(1, x_{0})}).
\]
The idea is to relate $\chi(R\Psi_{\bar{h}_{y_{0}}}(\mathbbm{1})_{(1, x_{0})})$ with the Euler characteristic of the Milnor fiber for $\bar{h}_{y_{0}}$ at $(1, x_{0})$. We will recall the definition of the Milnor fiber below.

Let $f$ be an analytic function germ at the origin of $\mathbb{C}^{n+1}$ with $f(0) = 0$. Let 
\[
B_{\epsilon} := \{z \in \mathbb{C}^{n+1} \, | \, |z_{0}|^{2} + \cdots |z_{n}|^{2} < \epsilon \}
\] 
and $S^{2n+1}_{\epsilon} = \partial \bar{B}_{\epsilon}$. 

\begin{teo}[Milnor \cite{Milnor:1968}]
\label{thm: Milnor}
\[
\varphi_{\epsilon}: S^{2n+1}_{\epsilon} \backslash f^{-1}(0) \longrightarrow \mathbb{S}^{1}, \quad z \mapsto f(z)/|f(z)|
\]
is a smooth locally trivial fibration for $\epsilon$ sufficiently small.
\end{teo}

\begin{definition}
For any $\theta \in \mathbb{S}^{1}$ and $\epsilon$ sufficiently small as in the above theorem, $\varphi^{-1}_{\epsilon}(\theta)$ is called the Milnor fiber of $f$ at the origin.
\end{definition}

To compare with the nearby cycle, we consider another description of the Milnor fiber. For $0 < \delta \ll \epsilon$, let
\[
D^{*}_{\delta} = \{t \in \mathbb{C} \, | \, 0 < |t| < \delta \}.
\]

\begin{teo}[L\^e \cite{Le:1977}]
\[
\psi: B_{\epsilon} \cap f^{-1}(D^{*}_{\delta}) \longrightarrow D^{*}_{\delta} 
\]
is a smooth locally trivial fibration for $0 < \delta \ll \epsilon$ both sufficiently small.
\end{teo}


\begin{prop}
For sufficiently small $\delta, \epsilon$ as in the above theorem and any $a \in D^{*}_{\delta}$, $\psi^{-1}(a)$ is diffeomorphic to the Milnor fiber of $f$ at the origin.
\end{prop}

\begin{proof}
Cf. \cite[Proposition 1.4]{Dimca:1992}.
\end{proof}

As a consequence, we can also define the Milnor fiber to be $\psi^{-1}(a)$. By \cite[Lemma 1.1.1]{Schurmann:2003},
\begin{align}
\label{eq: Milnor fiber}
\chi(R\Psi_{f}(\mathbbm{1})_{0}) \cong \chi(\psi^{-1}(a)).
\end{align}

Now let us assume $f(z)$ is a homogeneous polynomial of degree $N$. Following \cite{Dimca:1992}, we call $f^{-1}(1)$ the global Milnor fiber of $f$ at the origin.

\begin{prop}
$f^{-1}(1)$ is diffeomorphic to the Milnor fiber of $f$ at the origin.
\end{prop}

\begin{proof}
We can construct a homeomorphism $g$ such that the following diagram commutes
\[
\xymatrix{
S^{2n+1}_{\epsilon} \backslash f^{-1}(0) \ar[rr]^{g} \ar[dr]_{\varphi} & &  f^{-1}(\mathbb{S}^{1}) \ar[dl]^{f}   \\
& \mathbb{S}^{1} &
}
\]
Here
\[
g: z \mapsto |f(z)|^{-\frac{1}{N}} \cdot z.
\]
Note for any $z \in f^{-1}(\mathbb{S}^{1})$, there exists unique $t \in \mathbb{R}_{+}$ such that $t^{\frac{1}{N}} \cdot z \in S^{2n+1}_{\epsilon}$. One can check that this gives the inverse.
\end{proof}

\begin{rem}
\label{rk: global Milnor fiber}
Since the above diagram holds for all $\epsilon$, then in the case of homogeneous polynomials the Milnor fiber at the origin is homeomorphic to $\varphi^{-1}_{\epsilon}(\theta)$ for any $\epsilon$.
\end{rem}

\subsection{Application}

Let $f = \bar{h}_{y_{0}}((g_{1}, g_{2}), x)$, whose variables are denoted by $(X_{ij}, M_{ij}, N_{ij})$. After change of variables in Section~\ref{sec: vanishing cycle}, we denote the set of new variables by $z = (X'_{ij}, M'_{ij}, N'_{ij})$. Then we want to compute the Euler characteristic of the Milnor fiber of $f(z)$ at the point $z^{0} = (x'_{ij}, 0, 0)$.
We choose a small ball around this point
\[
B_{\epsilon} := \Big\{(X'_{ij}, M'_{ij}, N'_{ij}) \, | \, \sum_{ij}|X'_{ij} - x'_{ij}|^{2} + \sum_{ij}|M'_{ij}|^{2} + \sum_{ij}|N'_{ij}|^{2} \leqslant \epsilon \Big\}
\] 
such that
\[
\varphi_{\epsilon} = f/|f|: S_{\epsilon} \backslash f^{-1}(0) \longrightarrow \mathbb{S}^{1}
\]
is a smooth fibration as in Theorem~\ref{thm: Milnor}. Let $V$ be the subset of coefficient variables and $W$ be the subset of quadratic variables. Let 
\[
\bar{B}^{V}_{\epsilon} := \Big\{(X'_{ij}, M'_{ij}, N'_{ij})_{V} \, | \, \sum_{ij}|X'_{ij} - x'_{ij}|^{2} + \sum_{ij}|M'_{ij}|^{2} + \sum_{ij}|N'_{ij}|^{2}  \leqslant \epsilon \Big\}
\]
which is the projection of $\bar{B}_{\epsilon}$ onto the coefficient variables. It is a ball around the projection $z^{0}_{V}$ of $z^{0}$. For any $z_{V} \in \bar{B}^{V}_{\epsilon}$, we have
\[
|z_{V} - z^{0}_{V}|^2 \leqslant \epsilon.
\]
Let
\[
f_{z_{V}}(z_{W}) = f(z_{V}, z_{W}) \text{ for } z_{W} = (X'_{ij}, M'_{ij}, N'_{ij})_{W}.
\]
It is a quadratic form. Let $z^{0}_{W}$ be the projection of $z^{0}$ to the quadratic variables. Note $X'_{ij} \in W$ if and only if $i \in J$ or $j \in I$. By Lemma~\ref{lemma: projection to zero} and the formula~\ref{eq: change X}, we have $z^{0}_{W} = 0$. Let
\[
\varphi_{z_{V}, \epsilon - |z_{V} - z^{0}_{V}|^2} = f_{z_{V}}/|f_{z_{V}}|: S^{2|W| - 1}_{\epsilon - |z_{V} - z^{0}_{V}|^2} \backslash f^{-1}_{z_{V}}(0) \longrightarrow \mathbb{S}^{1}
\]
So we have a diagram
\[
\xymatrix{
S^{2|W| - 1}_{\epsilon - |z_{V} - z^{0}_{V}|^2} \backslash f^{-1}_{z_{V}}(0) \ar[r] \ar[d] & \mathbb{S}^{1} \ar@{=}[d]\\
S_{\epsilon}\backslash f^{-1}(0) \ar[r] \ar[d]^{\pi_{V}} & \mathbb{S}^{1} \\
\bar{B}^{V}_{\epsilon} &
}
\]
which gives a fiberation of the Milnor fiber $\varphi_{\epsilon}^{-1}(\theta)$ for some $\theta\in \mathbb{S}^{1}$ over a closed subset $C_{\epsilon}^{V}$ of $\bar{B}^{V}_{\epsilon}$. In view of Remark~\ref{rk: global Milnor fiber}, the fiber $\varphi_{z_{V}, \epsilon - |z_{V} - z^{0}_{V}|^2}^{-1}(\theta)$ is homeomorphic to the Milnor fiber of $f_{z_{V}}$ at the origin. By \eqref{eq: Milnor fiber},
\begin{align}
\label{eq: fiber}
\chi(\varphi_{z_{V}, \epsilon - |z_{V} - z^{0}_{V}|^2}^{-1}(\theta))=\chi(R\Psi_{f_{z_{V}}}(\mathbbm{1})_{0}).
\end{align}
Next we would like to compute the Euler characteristic of $\varphi_{z_{V}, \epsilon - |z_{V} - z^{0}_{V}|^2}^{-1}(\theta)$ through the nearby cycle. To do so, we need the following lemma.

\begin{lemma}
${\rm rank \, Hessian}(f_{z_{V}})_{0}$ is even.
\end{lemma}

\begin{proof}
We can divide the set $W$ of variables into two classes:
\[
W_{1} = \{ M'_{i_{\alpha}i_{\alpha'}} \, | \, (\alpha', \alpha) \in T \} \cup \{ M'_{ji} \, | \, j \notin I, i \in I \} \cup\{ X'_{ij} \, | \, i \notin J, j \in I \}
\]
and
\[
W_{2} = \{ X'_{j_{\alpha'}i_{\alpha}}  \, | \, (\alpha', \alpha) \in T \} \cup \{N'_{ji} \, | \, j \in J, i \notin J \} \cup \{ X'_{ij} \, | \, i \in J, j \notin I \}
\]
such that the Hessians of the restrictions of $f_{z_{V}}$ to variables in $W_{1}$ (resp. $W_{2}$) are both zero at $0$. Then the Hessian must be in the form 
\[
{\rm Hessian}(f_{z_{V}}) = \begin{pmatrix} 0 & B \\ B^{T} & 0\end{pmatrix}.
\]
So its rank is even.
\end{proof}

\begin{cor}\label{cor-nearby-cycle-zw}
$\chi(R\Psi_{f_{z_{V}}}(\mathbbm{1})_{0}) = 0$.
\end{cor}
\begin{proof}
Since $f_{z_{V}}$ is a quadratic form, we can change the coordinates such that
\[
f_{z_{V}}(u) = \sum_{i = 1}^{r} u^{2}_{i}
\]
where $r = {\rm rank}\, {\rm Hessian}(f_{z_{V}})$ is even by the previous lemma. By Sebastiani-Thom theorem,
\[
R\Phi_{f_{z_{V}}}[-1](\mathbbm{1})_{0} = \mathbb{C}[- r].
\]
Hence $\chi(R\Phi_{f_{z_{V}}}[-1](\mathbbm{1})_{0}) = (-1)^{r} = 1$. It follows $\chi(R\Psi_{f_{z_{V}}}(\mathbbm{1})_{0}) = 0$.
\end{proof}

\begin{prop}
\(
\chi(R\Psi_{f}(\mathbbm{1})_{z^{0}}) = 0.
\)
\end{prop}

\begin{proof} We have
\(
\chi(R\Psi_{f}(\mathbbm{1})_{z^{0}}) = \chi(\varphi^{-1}_{\epsilon}(\theta)),
\)
where $\theta \in \mathbb{S}^{1}$. The latter admits a fiberation over $C^{V}_{\epsilon} \subseteq \bar{B}^{V}_{\epsilon}$ with fibers $\varphi_{z_{V}, \epsilon - |z_{V} - z^{0}_{V}|^2}^{-1}(\theta)$. By the Leray spectral sequence
\[
H^{p}(C^{V}_{\epsilon}, (R^{q}\pi_{V})_{*} \mathbbm{1}) \Rightarrow H^{n}(\varphi^{-1}_{\epsilon}(\theta)),
\]
we have
\begin{align*}
\chi(\varphi^{-1}_{\epsilon}(\theta)) & = \sum_{p, q} (-1)^{p+q} {\rm dim} \, H^{p}(C^{V}_{\epsilon}, (R^{q}\pi_{V})_{*} \mathbbm{1}) = \sum_{q} (-1)^{q} \chi(H^{*}(C^{V}_{\epsilon}, (R^q\pi_{V})_{*} \mathbbm{1})) \\
& =  \sum_{q} (-1)^{q} \int_{C^{V}_{\epsilon}} \chi((R^q\pi_{V})_{*} \mathbbm{1}) = \int_{C^{V}_{\epsilon}} \chi((R\pi_{V})_{*} \mathbbm{1}).
\end{align*}
By Corollary~\ref{cor-nearby-cycle-zw} and \eqref{eq: fiber},
\[
\chi((R\pi_{V})_{*} \mathbbm{1})_{z_{V}} = \chi(\pi^{-1}_{V}(z_{V})) = \chi(\varphi_{z_{V}, \epsilon - |z_{V} - z^{0}_{V}|^2}^{-1}(\theta)) = 0.
\]
Hence, $\chi(\varphi^{-1}_{\epsilon}(\theta)) = 0$.

\end{proof}

This completes the proof of Conjecture~\ref{conj: vanishing cycle} for type $A_2$ quiver.

\section{Appendix}

\subsection{Regularity of stratification}\label{subsection-regularity}

Let $X$ be a closed subset of a smooth real manifold $M$ of dimension $m$. A smooth stratification of $X$ is a filtration 
\[
X_0 \subseteq \cdots \subseteq X_n \subseteq X 
\]
by closed subsets such that $X^{i} := X_i \backslash X_{i-1}$ is a smooth $i$-dim submanifold of $M$. For $j > i$, we say $X^{j}$ is $w$-regular over $X^i$ at $x \in X^{i} \cap \line{X}^{j}$ if there exists a neighborhood $U$ of $x$ in $M$ and constant $C$ such that in suitable local coordinates
\[
d(T_{x'}X^{i}, T_{x''}X^{j}) < C \| x'' - x' \|
\]
for all $x' \in U \cap X^{i}, x'' \in U \cap X^{j}$. Here we have chosen a norm $\| \cdot \|$ on $\mathbb{R}^{m}$ and identify $T_{x'}X^{i}, T_{x''}X^{j}$ with subspaces of $\mathbb{R}^{m}$. We define the distance between any two subspaces $V, W$ of $\mathbb{R}^{m}$ to be
\[
d(V, W) := {\rm sup}_{v \in V, \|v\| = 1} d(v, W)
\]
In our setting, we will take $M$ to be a complex variety and $X_{i}$ to be semialgebraic subsets of $M$. By \cite[Remark 4.1.9]{Schurmann:2003}, we have 
\[
\text{ $w$-regular } \Rightarrow \text{ Whitney $a$-regular, $b$-regular and $d$-regular} 
\]
So in order to apply \cite[Theorem 5.3.3]{Schurmann:2003}, it suffices to show $w$-regularity.

\begin{prop}
\label{prop: w-regular}
The stratification of $E_{V, \Omega}$ by $G_{V}$-orbits is $w$-regular.
\end{prop}

\begin{proof}
Let $S_{i}, S_{j}$ be $G_{V}$-orbits such that $\bar{S}_{j} \supseteq S_{i}$. For $x \in S_{i}$, we choose a small neighborhood $U$ of $x$ in $E_{V, \Omega}$ such that $\line{U \cap S_{i}} \subseteq S_{i}$ is compact. For any $x' \in U \cap S_{i}$, $T_{x'}S_{i} = [\mathfrak{g}_{V}, x']$. We fix a norm on $E_{V, \Omega}$. Let $N(x')$ be the subspace of $\mathfrak{g}_{V}$ orthogonal to the kernel of
\[
\mathfrak{g}_{V} \rightarrow E_{V, \Omega}, \quad h \mapsto [h, x'].
\] 
Let
\[
C(x') := {\rm sup} \{\| h \| \, | \, h \in N(x') \text{ and } \| [h, x']\| = 1 \}
\]
It is bounded by some positive constant $C$ on $\line{U \cap S_{i}}$. For $x'' \in U \cap S_{j}, x' \in U \cap S_{i}$, we can find $h \in \mathfrak{g}_{V}$ with $\|h\| \leqslant C$ such that 
\begin{align*}
d(T_{x'}S_{i}, T_{x''}S_{j}) & = d([h, x'], T_{x''}S_{j}) \leqslant d([h, x'], [h, x'']) \\
& = \|[h, x' - x'']\| \leqslant C' \cdot \|h\| \cdot \|x' - x''\| \leqslant C' C \cdot \|x' - x''\|
\end{align*}
for some positive constant $C'$ independent of $h, x', x''$.

\end{proof}

\subsection{A vanishing cycle calculation}\label{subsection-A-vanishing}

We will prove Proposition~\ref{prop: tangential data} following that of \cite[Theorem 6.7.5]{CMMB}. Let us recall the statement.

\begin{prop}
\label{prop: tangential data 1}
$R\Phi_{f}[-1](\mathbbm{1}_{S \times \widehat{S}})_{(x,y)} = \mathbb{C}[{\rm dim}\Lambda_{V} - {\rm dim} \widehat{S} - {\rm dim} S]$, where $f$ is the restriction of $\langle \,, \,\rangle$ to $S \times \widehat{S}$.
\end{prop}

First we need to make some preparations. 


\begin{lemma}
The function $f$ is singular over $T^{*}_{S}(E_{V, \Omega})_{{\rm reg}}$, i.e., $df|_{T^{*}_{S}(E_{V, \Omega})_{{\rm reg}}} = 0$. 
\end{lemma}

\begin{proof}
For any $(x,y) \in T^{*}_{S}(E_{V, \Omega})_{reg}$ and $u \in T_{x} S$, we have $df_{(x, y)}(u) = \langle u, y \rangle = 0$. Similarly, we have $df_{(x, y)}(v) = \langle x, v \rangle = 0$ for any $v \in T_{y} \widehat{S}$. This finishes the proof.
\end{proof}

Fix $(x, y) \in T^{*}_{S}(E_{V, \Omega})_{reg}$ and let $N \subseteq S \times \widehat{S}$ be a normal slice to $T^{*}_{S}(E_{V, \Omega})_{reg}$ at $(x, y)$. In particular, we require $N \cap T^{*}_{S}(E_{V, \Omega})_{reg} = (x, y)$. The key step is to show

\begin{prop}
\label{prop: Hessian}
The Hessian of $f$ at $(x, y)$ has rank ${\rm dim}\, S + {\rm dim}\, \widehat{S} - n = {\rm dim}\, N$. Moreover, the Hessian of $f|_{N}$ at $(x, y)$ is non-degenerate.
\end{prop}

We can pull back $f$ to the Lie algebras $\mathfrak{g}_{V} \times \mathfrak{g}_{V}$ of $G_{V} \times G_{V}$ near a neighborhood of $(x, y)$ as follows
\[
F(h_{1,} h_{2}) = \langle exp(h_{1})x, exp(h_{2})y \rangle : \mathfrak{g}_{V} \times \mathfrak{g}_{V} \rightarrow \mathbb{C} 
\]
It is easy to see that the rank of Hessian of $f$ at $(x, y)$ is the same as that of $F$ at $(0,0)$. 

\begin{lemma}
In a small neighborhood of $0$ in $\mathfrak{g}_{V}$, one can express
\begin{align*}
exp(h)x & = x + [h, x] + \frac{1}{2}[h, [h,x]] + \cdots + \frac{1}{n!}[h, [h, \dots, [h,x] \cdots ]] + \cdots \\
exp(h)y & = y + [h, y] + \frac{1}{2}[h, [h,y]] + \cdots + \frac{1}{n!}[h, [h, \dots, [h,y] \cdots ]] + \cdots
\end{align*}
\end{lemma}

\begin{proof}
For any $h \in \mathfrak{g}_{V}$ we can find real number $\delta > 0$, which only depends on the norm of $h$, such that the vector-valued function 
\[
G(t) := exp(th)x: [-\delta, \delta] \rightarrow V,
\] 
can be expressed as 
\[
G(t) = G(0) + G'(0)t + \frac{1}{2}G''(0)t^2 + \cdots +\frac{1}{n!}G^{(n)}(0)t^n+ \cdots
\]
Since
\[
G^{(n)}(t) = exp(th)[h, [h, \dots, [h,x] \cdots]]
\]
then
\begin{align*}
 exp(th)x & = x + [h, x]t + \frac{1}{2}[h, [h,x]]t^2 + \cdots + \frac{1}{n!}[h, [h, \dots, [h,x]\cdots]]t^n + \cdots \\
& = x + [th, x] + \frac{1}{2}[th, [th,x]] + \cdots + \frac{1}{n!}[th, [th, \dots, [th,x]\cdots]] + \cdots
\end{align*}
This proves the first equality. The second equality can be proved in the same way.

\end{proof}

Let us write $Z_{1}(h) = exp(h)x - (x + [h,x])$ and $Z_{2}(h) = exp(h)y - (y + [h, y])$. Then
\[
F(h_{1}, h_{2}) = \langle x + [h_{1},x], y + [h_{2}, y] \rangle  + \langle x + [h_{1},x], Z_{2}(h_{2}) \rangle + \langle Z_{1}(h_{1}), y + [h_{2}, y] \rangle + \langle Z_{1}(h_{1}), Z_{2}(h_{2}) \rangle
\]
The degree $2$ terms in the above expression can only come from $\langle [h_{1},x], [h_{2}, y]\rangle, \langle x, Z_{2}(h_{2}) \rangle$ and $\langle Z_{1}(h_{1}), y \rangle$. Therefore,
\[
{\rm Hessian}(F)_{(0,0)} = \begin{pmatrix} {\rm Hessian}( \langle Z_{1}(h_{1}), y \rangle )_{(0,0)} & B \\ B^{T} & {\rm Hessian}( \langle x, Z_{2}(h_{2}) \rangle )_{(0,0)} \end{pmatrix}
\]
where 
\[
B = \Big( \frac{\partial^2}{\partial h_{1} \partial h_{2}} \langle [h_{1},x], [h_{2}, y]\rangle  \Big)_{(0,0)}
\]
It is not hard to see that $B$ corresponds to the bilinear form 
\[
\langle [h_{1},x], [h_{2}, y] \rangle: \mathfrak{g}_{V} \times \mathfrak{g}_{V} \rightarrow \mathbb{C}
\]
after we identify $T_{(0,0)} (\mathfrak{g}_{V} \times \mathfrak{g}_{V})$ with $\mathfrak{g}_{V} \times \mathfrak{g}_{V}$. Since
\[
\langle [h_{1},x], [h_{2}, y] \rangle = \langle [y, [h_{1},x]], h_{2} \rangle,
\]
then the rank of the above bilinear form is ${\rm dim}\, [y, [\mathfrak{g}_{V}, x]]$. So we have shown

\begin{lemma}
\label{lemma: lower bound}
${\rm rank}\, {\rm Hessian}(F)_{(0,0)} \geqslant {\rm dim}\, [y, [\mathfrak{g}_{V}, x]]$.
\end{lemma}

Next, we would like to show

\begin{lemma}
\label{lemma: dim count}
${\rm dim}\, [y, [\mathfrak{g}_{V}, x]] = {\rm dim}\, S + {\rm dim}\, \widehat{S} - {\rm dim}\, \Lambda_{V}$.
\end{lemma}

\begin{proof}
It is easy to see that $[\mathfrak{g}_{V}, x] = T_{x}S$ and ${\rm Ker}[y, \cdot]|_{E_{V, \Omega}} = T^{*}_{\widehat{S}, y}E_{V, \bar{\Omega}}$. Since $(x, y)$ is regular, $T^{*}_{\widehat{S}, y}E_{V, \bar{\Omega}} \cap S \times \{y\}$ contains an open neighborhood of $(x, y)$ in $T^{*}_{\widehat{S}, y}E_{V, \bar{\Omega}}$. Hence $T^{*}_{\widehat{S}, y}E_{V, \bar{\Omega}} \subseteq T_{x}S$. So ${\rm dim}\, [y, [\mathfrak{g}_{V}, x]] = {\rm dim}\,  T_{x}S - {\rm dim}\, T^{*}_{\widehat{S}, y}E_{V, \bar{\Omega}} = {\rm dim}\,  T_{x}S - ( {\rm dim}\, \Lambda_{V} - {\rm dim}\, T_{y}\widehat{S}) = {\rm dim}\, S + {\rm dim}\, \widehat{S} - {\rm dim}\, \Lambda_{V}$.
\end{proof}

\begin{cor}
$T_{(x, y)}(T^{*}_{S}(E_{V, \Omega})) = \{(u, v) \in T_{x}S \times T_{y}\widehat{S} \, | \, [u, y] + [x, v] = 0\}$.
\end{cor}

\begin{proof}
For any $(u, v) \in T_{(x, y)}(\Lambda_{V})$, one can choose smooth $\alpha: [0, 1] \rightarrow V, \beta:[0, 1] \rightarrow E_{V, \bar{\Omega}}$ such that 
\[
\alpha(0) = x, \beta(0) = y, d\alpha (1) = u, d\beta(1) = v
\] 
and $(\alpha(t), \beta(t)) \in \Lambda_{V}$. Then $[\alpha(t), \beta(t)] = 0$. Differentiate it at $t = 0$:
\begin{align*}
0 & = \lim_{t \to 0} \frac{1}{t} \Big( [\alpha(t), \beta(t)] - [x, y] \Big) = \lim_{t \to 0} \frac{1}{t} \Big( [\alpha(t) - x, \beta(t)] + [x, \beta(t) - y] \Big) \\
& = [\lim_{t \to 0}  \frac{\alpha(t) - x}{t}, \lim_{t \to 0}  \beta(t)] + [x, \lim_{t \to 0}  \frac{\beta(t) - y}{t}]  = [u, y] + [x, v]
\end{align*}
It follows 
\[
T_{(x, y)}(\Lambda_{V}) \subseteq \{(u, v) \in T^{*}E_{V, \Omega} \, | \, [u, y] + [x, v] = 0\}.
\]
Since $(x, y)$ is regular, $T_{(x, y)}(\Lambda_{V}) = T_{(x, y)}(T^{*}_{S}(E_{V, \Omega})) \subseteq T_{x}S \times T_{y}\widehat{S}$. So 
\[
T_{(x, y)}(\Lambda_{V}) \subseteq \{(u, v) \in T_{x}S \times T_{y}\widehat{S} \, | \, [u, y] + [x, v] = 0\}
\]
and it is enough to show the dimension of the right hand side is equal to ${\rm dim}\, \Lambda_{V}$. Now let us consider 
\[
\varphi(u, v) = [u, y] + [x, v] : T_{x}S \times T_{y}\widehat{S} \rightarrow \mathfrak{g}_{V}
\]
Note ${\rm Ker}\, \varphi  = \{(u, v) \in T_{x}S \times T_{y}\widehat{S} \, | \, [u, y] + [x, v] = 0\}$. The image of $\varphi$ is $[[\mathfrak{g}_{V}, x], y] + [x, [\mathfrak{g}_{V}, y]] = [[\mathfrak{g}_{V}, x], y]$. By the previous lemma, ${\rm dim \, Im}\, \varphi = {\rm dim}\, S + {\rm dim}\, \widehat{S} - {\rm dim}\, \Lambda_{V}$. Hence ${\rm dim \, Ker}\, \varphi = {\rm dim}\, \Lambda_{V}$. This finishes the proof.

\end{proof}

Next we would like to compute the Hessian of $f$ at $(x ,y)$ in a different way. Let us choose local coordinates for a neighborhood $U$ of $(x, y)$ in $S \times \widehat{S}$ such that 
\[
U \cap T^{*}_{S}(E_{V, \Omega})_{reg} = \{\xi =  (\xi_{i})^{m}_{i = 1} \in U \, | \, \xi_{n+1} = \cdots = \xi_{m} = 0\}
\] 
and
\[
U \cap N = \{\xi =  (\xi_{i})^{m}_{i = 1} \in U \, | \, \xi_{1} = \cdots = \xi_{n} = 0\}.
\] 
By taking $U$ sufficiently small, we can assume $f|_{U}$ has analytic expansion
\[
f = \sum_{i_{1}, \cdots, i_{l}} c_{i_{1}, \cdots, i_{l}} \xi_{i_{1}}^{m_{i_{1}}} \cdots \xi_{i_{l}}^{m_{i_{l}}}. 
\]
Since $f$ is singular over $T^{*}_{S}(E_{V, \Omega})_{reg}$, the above expression can not have terms $\xi_{i}\xi_{j}$ with $i \leqslant n, j > n$. Note $f|_{T^{*}_{S}(E_{V, \Omega})_{reg}} = 0$. So 
\[
{\rm Hessian}(f|_{U})_{0} = \begin{pmatrix} 0 & 0 \\ 0 & {\rm Hessian}(f|_{N})_{0}\end{pmatrix}.
\]
It follows ${\rm rank \, Hessian}(f|_{U})_{0} = {\rm rank \, Hessian}(f|_{N})_{0} \leqslant {\rm dim} \, N$. Combining Lemma~\ref{lemma: lower bound} and Lemma~\ref{lemma: dim count}, we have proved Proposition~\ref{prop: Hessian}. Now we can prove Proposition~\ref{prop: tangential data 1}.

\begin{proof}
Let us choose local coordinates for a neighborhood $U$ of $(x, y)$ in $S \times \widehat{S}$ such that 
\[
U \cap T^{*}_{S}(E_{V, \Omega})_{reg} = \{\xi =  (\xi_{i})^{m}_{i = 1} \in U \, | \, \xi_{n+1} = \cdots = \xi_{m} = 0\}.
\] 
Let 
\[
N = \{\xi =  (\xi_{i})^{m}_{i = 1} \in U \, | \, \xi_{1} = \cdots = \xi_{n} = 0\},
\] 
which is a normal slice to $U \cap T^{*}_{S}(E_{V, \Omega})$ at $(x, y)$. Since $f$ vanishes and is singular on $T^{*}_{S}(E_{V, \Omega})_{reg}$, we can assume
\[
f(\xi) = \sum_{i, j > n} \alpha_{i, j}(\xi) \xi_{i}\xi_{j}
\]
By Proposition~\ref{prop: Hessian}, the Hessian of $f|_{N}$ is non-degenerate. So we can make a change of coordinates
\[
\xi'_{k} = \sum_{k, l} b_{kl}(\xi) \xi_{l}
\]
following the Grahm-Schmidt process such that 
\[
\Big(b_{k,l}(\xi) \Big)_{k,l} = \begin{pmatrix} I_{n} & 0 \\ 0 & B(\xi)\end{pmatrix}
\]
where $B(\xi)$ is upper-triangular with constant function $1$ on the diagonal, and
\[
f(\xi') = \sum_{i > n} \beta_{i}(\xi') \xi'^{2}_{i}.
\]
for $\beta_{i}(\xi')$ nonzero on a small neighborhood $W \subseteq U$. By choosing a branch of square roots, we can make a further change of coordinates by
\[
\xi''_{i} = \begin{cases} \xi'_{i} & \text{ if } i \leqslant n \\
\sqrt{\beta_{i}(\xi')} \, \xi'_{i} & \text{ if } i > n
 \end{cases}
\]
Then 
\[
f(\xi'') = \sum_{i > n} \xi''^{2}_{i}
\]
It follows from Sebastiani-Thom theorem that
\[
(R\Phi_{f}[-1](\mathbbm{1}_{U}))_{(x,y)} = \mathbb{C}[- {\rm dim}\, N].
\]
This finishes the proof.

\end{proof}

\bibliographystyle{plain}
\bibliography{biblio}

\end{document}